\documentclass[a4paper,12pt,twoside]{amsart}
\usepackage[english]{babel}
\usepackage{amssymb, amsmath, eucal}
\usepackage[all]{xy}
\usepackage{mathrsfs}
\usepackage[dvips]{graphics}

\newtheorem{The}{Theorem}[section]
\newtheorem{Lem}[The]{Lemma}
\newtheorem{Prop}[The]{Proposition}
\newtheorem{Cor}[The]{Corollary}
\newtheorem{Rem}[The]{Remark}

\newtheorem{Conj}[The]{Conjecture}

\newtheorem{Def}[The]{Definition}

\newtheorem*{ackn}{Acknowledgements}

\newcommand{\C}{\mathbb{C}}
\newcommand{\R}{\mathbb{R}}

\newcommand{\U}{\underline{u}}

\begin{document}
 \title[Parabolic complex Monge-Amp\`ere equations]{Degenerate complex Monge-Amp\`ere flows on  strictly pseudoconvex domains}

\setcounter{tocdepth}{1}

  \author{Do Hoang Son } 

\email{hoangson.do.vn@gmail.com }
 \date{\today}


\begin{abstract}
  We study the equation $\dot{u}=\log\det (u_{\alpha\bar{\beta}})+f(t,z,u)$ in  domains
  of $\C^n$. This equation has a close connection with the K\"ahler-Ricci flow. In this paper, we consider the case of the boundary conditions are smooth and the initial conditions are bounded.
\end{abstract}

\maketitle

\tableofcontents
\newpage

\section*{Introduction}
On K\"ahler manifolds, a K\"ahler-Ricci flow is an equation
\begin{equation}
\frac{\partial}{\partial t}\omega=-Ric(\omega),
\end{equation}
which starts from a K\"ahler metric. 
Here, $Ric(\omega)$ is the form associated  to the Ricci curvature of $\omega$, i.e., if 
$$
\omega=\frac{\sqrt{-1}}{2\pi}g_{i\bar{j}}dz^i\wedge d\overline{z^j},
$$ 
then
$$
Ric(\omega)=-\frac{\sqrt{-1}}{2\pi}(\partial_i\partial_{\overline{j}}\log\det g)
dz^i\wedge d\bar{z^j}.
$$
This flow was become a poweful tool of geometry. The theory of K\"ahler-Ricci flow
is well developed in the case of compact K\"ahler manifolds, see e.g. \cite{Cao85},
\cite{PS05}, \cite{ST07}, \cite{Zha09}, \cite{Tos10},
\cite{GZ13}, \cite{BG13}. It can be seen as the parabolic problem associated to an ``elliptic" problem
which would be the complex Monge-Amp\`ere equation. 

Monge-Amp\`ere equations and their generalizations have long been studied
in strictly pseudoconvex domains of $\mathbb C^n$, see for instance \cite{CKNS85}. 
This raises a natural
question: what is the behavior of the corresponding parabolic equation in the case of $\C^n$?

Let $\Omega$ be a bounded smooth strictly pseudoconvex domain of $\C^n$, i.e., there exists a smooth
strictly plurisubharmonic function $\rho$ defined on a bounded neighbourhood of $\bar{\Omega}$ such that 
$\Omega=\{\rho<0\}$ and $d\rho|_{\partial\Omega}\neq 0$.

 Let $T\in (0,\infty]$. We consider the equation
\begin{equation}\label{KRF}
\begin{cases}
\begin{array}{ll}
\dot{u}=\log\det (u_{\alpha\bar{\beta}})+f(t,z,u)\;\;\;&\mbox{on}\;\Omega\times (0,T),\\
u=\varphi&\mbox{on}\;\partial\Omega\times [0,T),\\
u=u_0&\mbox{on}\;\bar{\Omega}\times\{ 0\},\\
\end{array}
\end{cases}
\end{equation}
where $\dot{u}=\frac{\partial u}{\partial t}$,
 $u_{\alpha\bar{\beta}}=\frac{\partial^2 u}{\partial z_{\alpha}\partial\bar{z}_{\beta}}$,
 $u_0$ is a plurisubharmonic function in a neighbourhood of $\Omega$
 and $f$ is smooth in $[0,T)\times\bar{\Omega}\times\R$ and non increasing in the last variable.
 
 This equation has a close connection with the K\"ahler-Ricci flow. 
 There are some previous results. 
 If $u_0$ is continuous and $\varphi$ does not depend on the last variable,
  then \eqref{KRF} 
 admits a unique viscosity solution  \cite{EGZ14}. If $u_0$ is a smooth strictly 
 plurisubharmonic function in $\bar{\Omega}$,  $\varphi$ is smooth in $\bar{\Omega}\times 
 [0,T)$ and the compatibility conditions are satisfied, then 
  \eqref{KRF} admits a unique solution 
 $u\in C^{\infty}(\Omega\times (0,T))\cap C^{2;1}(\bar{\Omega}\times [0,T))$  
 \cite{HL10}; we state their result in detail as Theorem \ref{houli} in Section \ref{prelim}. 
 
  In this paper, we study the case where $\varphi$ is smooth and $u_0$ is merely bounded. 
The main result is the following:
\begin{The}
\label{app2}
 Let $\Omega$ be a bounded smooth strictly pseudoconvex 
 domain of $\C^n$ and $T\in (0,\infty ]$.
 Let $u_0$ be a bounded plurisubharmonic function defined on a neighbourhood 
 $\tilde{\Omega}$ of
 $\overline \Omega$. 
 Assume that $\varphi\in C^{\infty}(\bar{\Omega}\times [0,T))$ and $f
 \in C^{\infty}([0,T)\times \bar{\Omega}\times \R )$  satisfying
 \begin{itemize}
 \item[(i)] $f_u\leq 0.$
 \item[(ii)] $\varphi(z,0)=u_0(z)$ for $z\in\partial\Omega$.
 \end{itemize}
 Then there exists a unique function $u\in C^{\infty}(\bar{\Omega}\times (0,T))$ such that
 \begin{equation}
 \label{psh.dlchinh}
 u(.,t) \mbox{ is a strictly plurisubharmonic function on } \Omega \;\;
 \mbox{ for all } t\in (0,T),
 \end{equation}
 \begin{equation}\label{KRF.dlchinh}
\dot{u}= \log\det (u_{\alpha\bar{\beta}})+f(t,z,u) \mbox{ on } \Omega\times (0,T),
\end{equation}
\begin{equation}\label{bien.dlchinh}
u=\varphi \mbox{ on } \partial\Omega\times (0,T),
\end{equation}
\begin{equation}\label{initial.dlchinh}
\lim\limits_{t\rightarrow 0} u(z,t)= u_0(z)\;\;\forall z\in\bar{\Omega}.
 \end{equation}
 Moreover, $u\in L^{\infty}(\bar{\Omega}\times [0,T'))$ for any $0<T'<T$, and
 $u(.,t)$  also converges to $u_0$ in capacity when $t\rightarrow 0$.\\
  If $u_0\in C(\tilde{\Omega})$ then $u\in C(\bar{\Omega}\times [0,T))$.
 \end{The}
 Here, we say that $u(.,t)$  converges to $u_0$ in capacity if the convergence is uniform outside sets of arbitrarily small capacity.
 
 This improves the main result of \cite{HL10} in two directions: we do not need 
 smoothness of the initial data, and still have continuity when $t\to 0$;
  and we obtain the maximal possible regularity
 when $z$ tends to $\partial \Omega$, for fixed $t>0$. 
 
 Some techniques used in this paper are from the corresponding result 
 in the case of compact K\"ahler manifolds. 
 On a compact K\"ahler manifold,  results have been obtained in the 
 more general case where 
 $u_0$ has zero or even positive Lelong numbers. 
 We refer the reader to \cite{GZ13} and \cite{DL14} for the details.
 \begin{ackn}
 I am deeply grateful to Pascal Thomas and Vincent Guedj for many
inspiring discussions on the subject and encouragement me to write
down this paper. It is improved significantly thanks to their thorough reading and
 editing.  I also would like to thank Lu Hoang Chinh for very useful discussions about
 Proposition \ref{gra 2}. 
 \end{ackn}
\section{Strategy of the proof}

We fix some notation. We say that $u\in C^{2;1}(\bar{\Omega}\times [0,T))$ if 
 $u(.,t)\in C^2(\bar{\Omega})$ for any $t\in [0,T)$, $u(z,.)\in C^1([0,T))$ for any
 $z\in\bar{\Omega}$ and $\dot{u},u_{s_js_k}\in C(\bar{\Omega}\times [0,T))$ for
 $s_j,s_k\in\{x_1,y_1,...x_n,y_n\}$.\\

In order to prove Theorem \ref{app2}, we use an approximation process and
we first will need to prove  the following a priori estimates theorem:
\begin{The}
\label{Main} 
Let $\Omega$ be a bounded smooth strictly pseudoconvex domain of $\C^n$ and $T>0$. Let 
$\varphi\in C^{\infty}(\bar{\Omega}\times [0,T))$ and $f\in C^{\infty}([0,T)\times 
\bar{\Omega}\times \R )$ and let $u\in C^{\infty}(\Omega\times (0,T))\cap
C^{2;1}(\bar{\Omega}\times[0,T))$, strictly plurisubharmonic with respect to
$z$, be a solution of the equation
\begin{equation}\label{eqmain}
\dot{u}=\log\det (u_{\alpha\bar{\beta}})+f(t,z,u)\;\;\;\mbox{on}\;\; \Omega\times (0,T) .
\end{equation}
   Assume that 
\begin{equation}\label{uphi}
u|_{\partial\Omega\times [0,T)}=\varphi|_{\partial\Omega\times [0,T)},
\end{equation}
\begin{equation}\label{cu}
\sup |u(z,0)|\leq C_u,
\end{equation}
\begin{equation}\label{fu}
f_u(t,z,u)\leq 0\;\;\;\forall (t,z,u)\in (0,T)\times \Omega\times \R,
\end{equation}
\begin{equation}\label{cf}
\|f\|_{C^2((0,T)\times\Omega \times \R)}\leq C_f,
\end{equation}
\begin{equation}\label{cphi}
\|\varphi\|_{C^4(\Omega\times (0,T))}\leq C_{\varphi}.
\end{equation}
Then   there exists $M_0=M_0(\Omega, T, C_u, C_{\varphi},C_f)$ and for any $0<\epsilon<T$ there exists $C=C(\Omega, \epsilon, T, C_u, C_{\varphi},C_f)$ such that
\begin{center}
$|u|\leq M_0\;\;$ on $\;\; \Omega\times (0,T),$\\
$|\nabla u|+|\dot{u}|+\Delta u\leq C\;\;$ on $\;\;\Omega\times (\epsilon,T).$
\end{center}
\end{The}
\begin{Rem}
In the theorem above, we denote
$$
\|\varphi\|_{C^k(\Omega\times (0,T))}
=\sum\limits_{|j|+2l\leq k}\sup\limits_{\Omega\times (0,T)}
|D^j_s D^l_t\varphi|,
$$
$$
\|f\|_{C^k((0,T)\times\Omega\times\R))}
=\sum\limits_{j_1+|j_2|+j_3\leq k}\sup |D_t^{j_1}D_s^{j_2}D_u^{j_3}f|,
$$
where $s=(s_1,...,s_{2n})=(x_1,y_1,...,x_n,y_n)$.
\end{Rem}
For the proof of Theorem \ref{app2}, the strategy is as follows.

\begin{itemize}
\item[+] Construct the solutions 
$u_m\in C^{\infty}(\Omega\times (0,T))\cap C^{2;1}(\bar{\Omega}\times [0,T))$ of 
\eqref{KRF.dlchinh} such that $u_m|_{\bar{\Omega}\times\{0\}}$ and 
$u_m|_{\partial\Omega\times (0,T)}$ converge pointwise, respectively, to $u_0$ 
and $\varphi|_{\partial\Omega\times (0,T)}$. 
We also ask that the $u_m$ be uniformly bounded and
$u_m|_{\partial\Omega\times (\epsilon_m,T)}=\varphi|_{\partial\Omega\times (\epsilon_m,T)}$ 
for some
$\epsilon_m\searrow 0$.
\item[+] Use the a priori estimates to prove
$$
\|u_m\|_{C^2(\bar{\Omega}\times (\epsilon,T'))}\leq C_{\epsilon,T'}
$$
for any $0<\epsilon<T'<T$, where $C_{\epsilon,T'}>0$ is independent of $m$.
\item[+] Use $C^{2,\alpha}$ estimates and to prove
$$
\|u_m\|_{C^k(\bar{\Omega}\times (\epsilon,T'))}\leq C_{k,\epsilon,T'}
$$
for any $0<\epsilon<T'<T$ and $k>0$,
where $C_{k,\epsilon,T'}>0$ is independent on $m$. The 
$C^{2,\alpha}$ estimates  and the $C^{k,\alpha}$ regularity  will be mentioned in  section 5.
\item[+] By Ascoli's theorem, there exists a subsequence of $\{u_m\}$, denoted also by  $\{u_m\}$, 
and $u\in C^{\infty}(\bar{\Omega}\times (0,T)$ such that
$$
u_m\stackrel{C^k(\bar{\Omega}\times (\epsilon,T'))}{\longrightarrow}u.
$$
Then, $u$ satisfies \eqref{psh.dlchinh}, \eqref{KRF.dlchinh} 
and \eqref{bien.dlchinh}.
\item[+] Use Comparison principle to prove \eqref{initial.dlchinh}.
\item[+]Finally, we prove the uniqueness of $u$. 
\end{itemize}
We will study some important tools  before we prove Theorem \ref{app2}. In Section \ref{prelim}, we 
introduce some basic results about parabolic complex Monge-Amp\`ere equations. 
In Sections \ref{order1} and \ref{order2}, we prove
the a priori estimates theorem (Theorem \ref{Main}). In Section \ref{C2a}
 we establish the $C^{2,\alpha}$ estimate
needed to solve our problem.  Finally in Section \ref{pfmain} we prove Theorem \ref{app2}. 

\section{Preliminaries}
\label{prelim}
\subsection{Hou-Li theorem}
\begin{flushleft}
The Hou-Li theorem states
that equation \eqref{KRF} has a unique solution when the conditions are good 
enough. We will use it in Section \ref{pfmain} to obtain smooth solutions
to an approximating problem, to which we then will apply the a priori estimates
from Theorem \ref{Main}.  
\end{flushleft}
We first need the notion of subsolution.
\begin{Def}
\label{HLsubsol}
A function $\underline{u}\in C^{\infty}(\bar{\Omega}\times [0,T))$
is called a \emph{subsolution} of the equation \eqref{HLKRF} if and only if
\begin{equation}
\label{subsolu.houli}
\begin{cases}
\U(.,t) \mbox{is a strictly plurisubharmonic function,}\\
\dot{\U}\leq \log\det (\U)_{\alpha\bar{\beta}}+f(t,z,\U),\\
\U|_{\partial\Omega\times (0,T)}=\varphi|_{\partial\Omega\times (0,T)},\\
\U(.,0)\leq u_0 .
\end{cases} 
\end{equation}

\end{Def}
\begin{The}\label{houli}
Let $\Omega\subset\C^n$ be a bounded domain with smooth boundary. Let $T\in (0,\infty]$. Assume
that
\begin{itemize}
\item $\varphi$ is a smooth function in $\bar{\Omega}\times [0,T)$.
\item $f$ is a smooth function in $[0,T)\times\bar{\Omega}\times\R$ non increasing in the
lastest variable.
\item $u_0$ is a smooth strictly plurisubharmonic funtion in a neighborhood of $\Omega$.
\item $u_0(z)=\varphi (z,0),\;\forall z\in\partial\Omega$.
\item The compatibility condition is satisfied, i.e.
$$
\dot{\varphi}=\log\det (u_0)_{\alpha\bar{\beta}}+f(t,z,u_0),
\;\;\forall (z,t)\in \partial\Omega\times \{ 0 \}.
$$
\item There exists a subsolution to the equation \eqref{HLKRF}.
\end{itemize} 
Then there exists a unique solution 
$u\in C^{\infty}(\Omega\times (0,T))\cap C^{2;1}(\bar{\Omega}\times [0,T))$ of the equation
\begin{equation}
\label{HLKRF}
\begin{cases}
\begin{array}{ll}
\dot{u}=\log\det (u_{\alpha\bar{\beta}})+f(t,z,u)\;\;\;&\mbox{on}\;\Omega\times (0,T),\\
u=\varphi&\mbox{on}\;\partial\Omega\times [0,T),\\
u=u_0&\mbox{on}\;\bar{\Omega}\times\{ 0\}.\\
\end{array}
\end{cases} 
\end{equation}
\end{The}
\begin{Rem}
\begin{itemize}
\item[(i)]There is a corresponding result in the case of a compact K\"ahler manifold.
 On the compact  K\"ahler manifold $X$, we
must assume that $0<T<T_{max}$, where $T_{max}$ depends on $X$. In the case of domain 
$\Omega\subset\C^n$, we can assume that $T=+\infty$ if $\varphi, \U$ are defined on 
$\bar{\Omega}\times [0,+\infty)$ and $f$ is defined on $[0,+\infty)\times\bar{\Omega}\times\R$.
\item[(ii)] If $\Omega$ is a bounded smooth strictly pseudoconvex domain of $\C^n$ then 
one can prove that a subsolution always exists, and so Theorem 
\ref{houli} does not need the additional assumpation of existence of a subsolution.
\end{itemize}

\end{Rem}
\subsection{Maximum principle}
\begin{flushleft}
The following maximum principle is a basic tool to
establish upper and lower bounds in the sequel (see \cite{BG13} and \cite{IS13} for the proof).
\end{flushleft}
\begin{The}\label{max prin}
Let $\Omega$ be a bounded domain of $C^n$ and $T>0$. 
Let $\{\omega_t\}_{0<t<T}$ be a continuous family of continuous 
  positive definite Hermitian forms on $\Omega$.  Denote by $\Delta_t$ the Laplacian with 
 respect to $\omega_t$:
 $$
 \Delta_t f=\dfrac{n\omega_t^{n-1}\wedge dd^cf}{\omega_t^n},\;\forall f\in C^{\infty}(\Omega).
 $$
   Suppose that
$H \in C^{\infty}(\Omega \times (0,T)) \cap C(\bar{\Omega}\times [0,T))$ and satisfies\\
\begin{center}
$(\frac{\partial}{\partial t}-\Delta_t)H \leq 0 \:$ 
or
$\: \dot{H}_t \leq \log\dfrac{(\omega_t+dd^c H_t)^n}{\omega_t^n}$.
\end{center}
Then $\sup\limits_{\bar{\Omega}\times [0,T)} H = \sup\limits_{\partial_P(\Omega \times
[0,T))}H$. Here we denote $\partial_P(\Omega\times (0,T))=\partial\Omega\times (0,T)
\cup \bar{\Omega}\times\{ 0\}$.
\end{The}
\begin{Cor}(Comparison principle)\label{compa} 
Let $\Omega$ be a bounded domain of $\C^n$ and $T\in (0,\infty]$. Let
 $u,v\in C^{\infty}(\Omega\times (0,T))\cap C(\bar{\Omega}\times [0,T))$ 
satisfying
\begin{itemize}
\item $u(.,t)$ and $v(.,t)$ are strictly plurisubharmonic functions for any $t\in [0,T)$,
\item $\dot{u}\leq \log\det (u_{\alpha\bar{\beta}})+f(t,z,u),$
\item $\dot{v}\geq \log\det (v_{\alpha\bar{\beta}})+f(t,z,v),$
\end{itemize}
where $f\in C^{\infty}([0,T)\times\bar{\Omega}\times\R)$ is non increasing in the last 
variable.\\
Then $\sup\limits_{\Omega\times (0,T)}(u-v)\leq max\{ 0, 
\sup\limits_{\partial_P(\Omega\times (0,T))}(u-v)\}$.
\end{Cor}
\begin{Cor} Let $\Omega$ be a bounded domain of $\C^n$ and $T\in (0,\infty]$. We denote
by  $L$  a operator on $C^{\infty}(\Omega\times (0,T))$ satisfying
$$
L(f)=\dfrac{\partial f}{\partial t}-\sum a_{\alpha\bar{\beta}}
\dfrac{\partial^2f}{\partial z_{\alpha}\partial \bar{z}_{\beta}}-b.f,
$$
where $a_{\alpha\bar{\beta}}, b\in C(\Omega\times (0,T))$, $(a_{\alpha\bar{\beta}}(z,t))$
are positive definite Hermitian matrices and $b(z,t)<0$.\\
Assume that  $\phi\in C^{\infty}(\Omega\times (0,T))\cap C(\bar{\Omega}\times [0,T))$ 
satisfies\\
$$L(\phi)\leq  0.$$
Then $\phi\leq max(0,\sup\limits_{\partial_P(\Omega\times (0,T))}\phi).$
\end{Cor}
\subsection{The Laplacian inequalities}
\begin{flushleft}
We shall need two standard auxiliary results (see \cite{Yau78}, \cite{Siu87} for a proof).
\end{flushleft}

\begin{The}
\label{lap 1}
Let $\omega_1,\omega_2$ be positive $(1,1)$-forms on a complex manifold $X$.Then\\
$$
n\left(\dfrac{\omega_1^n}{\omega_2^n}\right)^{1/n}\leq tr_{\omega_2}(\omega_1)
\leq n\left(\dfrac{\omega_1^n}{\omega_2^n}\right)(tr_{\omega_1}(\omega_2))^{n-1},
$$
where $tr_{\omega_1}(\omega_2)=\dfrac{n\omega_1^{n-1}\wedge\omega_2}{\omega_1^n}$.
\end{The}
\begin{The}\label{lap 2}
Let $\omega, \; \omega '$ be two K\"ahler forms on a complex manifold $X$. If the holomorphic bisectional curvature of $\omega$ is bounded below by a constant $B\in\R$ on $X$,then\\
$$
\Delta_{\omega '}\log tr_{\omega}(\omega ')\geq -\frac{tr_{\omega}Ric(\omega ')}{tr_{\omega}(\omega ')}+B\, tr_{\omega '} (\omega),
$$
where $Ric(\omega')$ is the form associated  to the Ricci curvature of $\omega'$.
\end{The}
\begin{Rem}
Applying Theorem \ref{lap 2} for $\omega=dd^c|z|^2$ and  $\omega'=dd^cu$, we have\\
$$
\sum u^{\alpha\bar{\beta}}(\log \Delta u)_{\alpha\bar{\beta}}\geq\dfrac{\Delta\log\det (u_{\alpha\bar{\beta}})}{\Delta u}.
$$
\end{Rem}
\subsection{Construction of subsolutions}\label{subsolution}
\begin{flushleft}
We give a first construction which will be used in the proof of Theorem \ref{Main}.
First we need a notion of subsolution weaker than the one in Definition \ref{HLsubsol}.
\end{flushleft}
\begin{Def}
We say that a function $\underline{u}\in C^{\infty}(\bar{\Omega}\times [0,T))$ is a 
\emph{subsolution} of the equation \eqref{eqmain} if
$$
\dot{\underline{u}}\leq \log\det (\underline{u}_{\alpha\bar{\beta}})+f(t,z,\underline{u}).
$$
\end{Def}
We will construct subsolutions of \eqref{eqmain} in order to prove some estimates on the boundary.

Let $\rho\in SPSH(\bar{\Omega})\cap C^{\infty}(\bar{\Omega})$ be a function which defines $\Omega$. We also assume that $\inf\rho=-1$. 
Let $\zeta\in C^{\infty}(\R)$ such that $0\leq\zeta\leq 1$, $\zeta|_{[0,1]}=1$ and
 $\zeta|_{[2,\infty)}=0$. 
 
Let $\varphi$ and $u_0$ be as in Theorem \ref{Main}.
For any $m>0$, we denote  the function $\varphi_m\in C^{\infty}
 (\bar{\Omega}\times [0,T))$ by the formula
$$\varphi_m=\varphi-Osc(u_0)\cdot\zeta(mt).$$
Then there exists $M_m>0$ depending on $\rho, T, C_u, C_{\varphi}, C_f$ such that the function
$\underline{u}_m=\varphi_m+M_m\rho$ satisfies
\begin{center}
$\dot{\U}_m\leq \log\det(\U_m )_{\alpha\bar{\beta}}+f(t,z,\U_m)$ on $\Omega\times (0,T),$\\
$dd^c(\U_m)\geq dd^c|z|^2$ on $\Omega\times [0,T).$\\
\end{center}
Then $\U_m$ is a subsolution of \eqref{eqmain}. Moreover,
\begin{center}
$\U_m|_{\partial_P(\Omega\times (0,T))}\leq u|_{\partial_P(\Omega\times (0,T))},$\\
$\U_m|_{\partial\Omega\times (\frac{2}{m},T)}=
\varphi|_{\partial\Omega\times (\frac{2}{m},T)}.$
\end{center}
By the maximum principle, we have
\begin{center}
$\underline{u}_m\leq u$ on $\Omega\times (0,T)$.\\
\end{center}
In the next two sections, we will prove Theorem \ref{Main}. For convenience, we define
 an operator $L$ on $C^{\infty}(\Omega\times (0,T))$  by the formula
 \begin{equation}
 \label{L}
 L(\phi)=\dot{\phi}-\sum u^{\alpha\bar{\beta}}\phi_{\alpha\bar{\beta}}-f_u(t,z,u)\phi,
 \end{equation}
 where $u$ is the function in Theorem \ref{Main} and $(u^{\alpha\bar{\beta}})$ is the transpose of 
 inverse matrix of Hessian matrix  $(u_{\alpha\bar{\beta}})$.
 
\section{Order 1 a priori estimates}
\label{order1}
In this section, we will estimate $u$, $\dot{u}$ and $|\nabla u|$.\\
Clearly,\\
\begin{center}
$\underline{u}_1\leq u \leq \sup\limits_{\partial\Omega\times (0,T)}\varphi\;$ on $\Omega\times (0,T)$.\\
\end{center}
Then \\
\begin{center}
$-M_1 -2\sup|\varphi|-C_u\leq u(z,t) 
\leq\sup\limits_{\partial\Omega\times (0,T)}\varphi,$\\
\end{center}
where $M_1$ is the constant defined in \ref{subsolution}.
Let $C_1=M_1+2C_{\varphi}+C_u$ , we obtain 
\begin{equation}
\label{sup u}
\sup |u|\leq C_1.
\end{equation}
 \subsection{Bounds on $\bf{\dot{u}}$}
 \begin{Prop} \label{t}
There exists $C_2>0$ depending only on $T,C_f, C_1$ such that\\
\begin{center}
$t|\dot{u}|\leq C_2$ on $\Omega\times (0,T).$
\end{center}
 \end{Prop}
 \begin{proof}
 Take $L$ as in \eqref{L}, then
 $$
 L(t\dot{u}-u)=t\ddot{u}-t\sum u^{\alpha\bar{\beta}}\dot{u}_{\alpha\bar{\beta}}+n
 -(t\dot{u}-u)f_u(t,z,u).
 $$
  By equation \eqref{eqmain}, we have\\
 $$t\ddot{u}=t\sum u^{\alpha\bar{\beta}}\dot{u}_{\alpha\bar{\beta}}+t.f_t(t,z,u)+t\dot{u}.f_u(t,z,u).$$\\
 Then
 $$
 -C^{'}_2\leq L(t\dot{u}-u)=n+t.f_t(t,z,u)+u.f_u(t,z,u)\leq C^{'}_2,
 $$
 where $C^{'}_2=n+C_f(T+C_1)>0$.\\
 Since $L(t\dot{u}-u-C^{'}_2t)\leq 0$ and 
 $L(t\dot{u}-u +C^{'}_2t)\geq 0$,
 by the maximum principle, we obtain
 $$
 t\dot{u}-u-C^{'}_2t\leq \sup\limits_{\partial_P(\Omega\times (0,T))}(t\dot{u}-u-C^{'}_2t)
 \leq (C_{\varphi}+C^{'}_2)T+C_1 ,
 $$
 $$
 t\dot{u}-u +C^{'}_2t\geq \inf\limits_{\partial_P(\Omega\times (0,T))}(t\dot{u}-u
  +C^{'}_2t)\geq -(C_{\varphi}+C^{'}_2)T-C_1.
  $$
 Thus
$t|\dot{u}|\leq C_2$ on $\Omega\times (0,T)$,
where $C_2=(C_{\varphi}+2C^{'}_2)T+2C_1$.
 \end{proof}
 
 \subsection{Gradient estimates}
 \begin{Prop} 
 \label{gra 1}
 Let $m>\frac{2}{T}$. Then there exists $C_3=C_3(\Omega, M_m,C_{\varphi})>0$ such that\\
 \begin{center}
 $|\nabla u |\leq C_3$ on $\partial\Omega\times (\frac{2}{m},T).$
 \end{center}
 \end{Prop}
 \begin{proof}
 Let $h\in C^{\infty}(\bar{\Omega}\times [0,T))$ be a spatial harmonic function (i.e.
 harmonic with respect to $z$) satisfying
 \begin{center}
 $h=\varphi$ on $\partial\Omega\times [0,T).$
 \end{center}
 Then taking $\underline{u}_m$ as \ref{subsolution} , we have
 \begin{center}
 $\underline{u}_m\leq u\leq h$ on $\Omega\times (\frac{2}{m},T),$\\
 $\underline{u}_m=u=h=\varphi$ on $\partial\Omega\times (\frac{2}{m},T).$
 \end{center}
 Hence 
 \begin{center}
 $|\nabla (u-\U_m) |\leq |\nabla (h-\U_m)|$ on $\partial\Omega\times (\frac{2}{m},T).$
 \end{center}
 Thus
 \begin{center}
 $|\nabla u|\leq |\nabla \U_m|+ |\nabla (h-\U_m)|\leq C_3$ on $\partial\Omega\times (\frac{2}{m},T),$\\
 \end{center}
 where $C_3 > 0$ depends only on $\Omega, C_{\varphi}, M_m$.\\
 \end{proof}
 
 \begin{Prop}\label{gra 2}
 Assume that $m, C_3$ satisfy Proposition $\ref{gra 1}$ and $\frac{2}{m}<\epsilon<T$. 
 Then there exists $C_4=C_4(\Omega, m,\epsilon, T, C_f,C_1,C_2,C_3)>0$ such that
 \begin{center}
 $|\nabla u |\leq C_4$ on $\Omega\times (\epsilon,T).$
 \end{center}
 \end{Prop}
 
 \begin{proof}
 We will use the technique of Blocki as in \cite{Blo08}. In this proof only, we denote\\
 \begin{center}
 $g(t)=n\log (t-\dfrac{2}{m})$,\\
 $\gamma (u)=Au-Bu^2\;\;$ where $\;\; A=\dfrac{1}{10C_1}, B=\dfrac{1}{20C_1^2}$,\\
 $\eta=\dfrac{1}{4(\mbox{{\rm diam}} \Omega)^2}$,\\
 $\phi=\log|\nabla u|^2+\gamma (u)+g(t) +\eta|z|^2$,
 \end{center}
 and we assume that $0\in\Omega$.
 
  Let $\epsilon<T'<T$, we will prove that 
 $$
 \sup\limits_{\Omega\times (\frac{2}{m},T')}\phi\leq \tilde{C}_4,
 $$
 where $\tilde{C}_4$ depends on $\Omega,C_1,C_2,C_3,m, T, C_f$.
 
 Notice that the hypotheses and previous bounds on $|u|$ imply that,
 for $ t\in (\frac{2}{m},T'),$
\begin{equation}
\label{philessu}
 \exp \phi(z,t) \le |\nabla u (z,t)|^2 (t-\dfrac{2}{m})^n 
 \exp\left( \max_{\Omega\times (\frac{2}{m},T')} \gamma (u) +\eta \max_\Omega |z| \right)
 \le C |\nabla u|^2, 
\end{equation}
and in a similar way
$$
 |\nabla u (z,t)|^2 \le C (\epsilon-\dfrac{2}{m})^{-n}  \exp \phi(z,t)
 \le C_\epsilon  \exp \phi(z,t),  \quad t\in (\epsilon,T'),
 $$
 so the bound on $\phi$ yields a bound on $|\nabla u (z,t)|$.
 
Suppose that
 $$
 \sup\limits_{\Omega\times (\frac{2}{m},T')}\phi=\phi (z_0,t_0).
 $$
 
 By an orthogonal change of coordinates, we can assume that $(u_{\alpha\bar{\beta}}(z_0,t_0))$ is 
 diagonal. For convenience, we denote $u_{\alpha\bar{\alpha}}(z_0,t_0)=\lambda_{\alpha}$.
 
 We also denote by $\mathcal{L}$ the operator
 $$
 \mathcal{L}=\dfrac{\partial}{\partial t}-\sum u^{\alpha\bar{\beta}}
 \dfrac{\partial^2}{\partial z_{\alpha}\partial\bar{z}_{\beta}}.
 $$
 
  If $|\nabla u|^2(z_0,t_0)\leq C$, by \eqref{philessu}, we are done. In particular,
if $z_0 \in \partial \Omega$, we know that $ |\nabla u (z,t)|$ is bounded.
So we may restrict attention to the case 
 where  $|\nabla u|^2(z_0,t_0)>1$
 and $(z_0,t_0)\in\Omega\times (\dfrac{2}{m},T']$. 
 Then
 $\mathcal{L}(\phi)|_{(z_0,t_0)}\geq 0$.

 We compute
 \begin{flushleft}
 $\begin{array}{ll}
 \mathcal{L}(\phi)&=\mathcal{L}(\log|\nabla u|^2)+\gamma'(u).\dot{u}+g'(t)
 -\gamma'(u)\sum u^{\alpha\bar{\beta}}u_{\alpha\bar{\beta}}\\[8pt]
 &-\gamma''(u)\sum u^{\alpha\bar{\beta}}u_{\alpha}u_{\bar{\beta}}
 -\eta\sum u^{\alpha\bar{\alpha}}\\[8pt]
 &=\mathcal{L}(\log|\nabla u|^2)+\gamma'(u).(\dot{u}-n)+g'(t)\\[8pt]
 &-\gamma''(u)\sum u^{\alpha\bar{\beta}}u_{\alpha}u_{\bar{\beta}}
 -\eta\sum u^{\alpha\bar{\alpha}}.\\[8pt]
 \end{array}$
 \end{flushleft}
 When $|\nabla u|\neq 0$, we have
 \begin{flushleft}
 $\begin{array}{ll}
 (\log |\nabla u|^2)_{\alpha\bar{\beta}}
 &=\dfrac{|\nabla u|^2_{\alpha\bar{\beta}}}{|\nabla u|^2}
 -\dfrac{|\nabla u|^2_{\alpha}|\nabla u|^2_{\bar{\beta}}}{|\nabla u|^4}\\
 &=\dfrac{\langle\nabla u_{\alpha\bar{\beta}}, \nabla u\rangle}{|\nabla u|^2}
 +\dfrac{\langle\nabla u, \nabla u_{\beta\bar{\alpha}}\rangle}{|\nabla u|^2}
 +\dfrac{\langle\nabla u_{\alpha}, \nabla u_{\beta}\rangle}{|\nabla u|^2}\\
 &+\dfrac{\langle\nabla u_{\bar{\beta}}, \nabla u_{\bar{\alpha}}\rangle}{|\nabla u|^2}
 -\dfrac{|\nabla u|^2_{\alpha}|\nabla u|^2_{\bar{\beta}}}{|\nabla u|^4}.
 \end{array}$
 \end{flushleft}
 \begin{flushleft}
 $\begin{array}{ll}
 \mathcal{L}(\log|\nabla u|^2)&=
 \dfrac{\langle\nabla\dot{u},\nabla u\rangle
 -\sum\langle u^{\alpha\bar{\beta}}\nabla u_{\alpha\bar{\beta}},\nabla u\rangle}{|\nabla u|^2}
 +\dfrac{\langle\nabla u,\nabla \dot{u}\rangle
 -\sum \langle\nabla u,u^{\beta\bar{\alpha}}\nabla u_{\beta\bar{\alpha}}\rangle}{|\nabla u|^2}\\[8pt]
 &-\sum u^{\alpha\bar{\beta}}\dfrac{\langle\nabla u_{\alpha},\nabla u_{\beta}\rangle+
 \langle\nabla u_{\bar{\beta}},\nabla u_{\bar{\alpha}} \rangle}{|\nabla u|^2}
 +\sum u^{\alpha\bar{\beta}}\dfrac{(|\nabla u|^2)_{\alpha}(|\nabla u|^2)_{\bar{\beta}}}{|\nabla u|^4}.\\
 \end{array}$
 \end{flushleft}
 We have, by \eqref{eqmain},
  \begin{flushleft}
 $\begin{array}{ll}
 \mathcal{L}(\log|\nabla u|^2)|_{(z_0,t_0)}
 &= 2Re\left(\dfrac{\langle\nabla u, \nabla f\rangle}{|\nabla u|^2}\right)
 +2f_u(t,z,u) |\nabla u|^2
 -\sum\dfrac{|\nabla u_k|^2+|\nabla u_{\bar{k}}|^2}{\lambda_k|\nabla u|^2}\\[8pt]
 &+\sum\dfrac{(|\nabla u|^2)_k(|\nabla u|^2)_{\bar{k}}}{\lambda_k |\nabla u|^4}\\[8pt]
 &\leq \dfrac{2|\nabla f|}{|\nabla u|}+\sum\dfrac{(|\nabla u|^2)_k(|\nabla u|^2)_{\bar{k}}}{\lambda_k |\nabla u|^4}.\\[8pt]
 \end{array}$
 \end{flushleft}
 Hence, there exists $C_4^{'}=C_4^{'}(m,C_1,C_2,C_f)$ such that\\
 \begin{flushleft}
 $\begin{array}{ll}
 \mathcal{L}(\phi)|_{(z_0,t_0)}&\leq C_4^{'}+g'(t)-\gamma''(u)\sum\dfrac{|u_k|^2}{\lambda_k}-\eta\sum\dfrac{1}{\lambda_k}
 +\sum\dfrac{(|\nabla u|^2)_k(|\nabla u|^2)_{\bar{k}}}{\lambda_k |\nabla u|^4}.\\[8pt]
 \end{array}$
 \end{flushleft}
 By the condition $\frac{\partial\phi}{\partial z_k}(z_0,t_0)=0$, we have\\
 $$
 \dfrac{(|\nabla u|^2)_k(|\nabla u|^2)_{\bar{k}}}{ |\nabla u|^4}
 =|\gamma'(u)u_k+\eta\bar{z}_k|^2\leq 2(\gamma'(u))^2|u_k|^2+2\eta^2|z_k|^2
 \leq 2(\gamma'(u))^2|u_k|^2+\dfrac{\eta}{2},
 $$
 where $(z,t)=(z_0,t_0).$\\
 Then
 \begin{flushleft}
 $\begin{array}{ll}
 0\leq\mathcal{L}(\phi)|_{(z_0,t_0)}&\leq C_4^{'}+g'(t)
 +(2(\gamma'(u))^2-\gamma''(u))\sum\dfrac{|u_k|^2}{\lambda_k}-\dfrac{\eta}{2}\sum\dfrac{1}{\lambda_k}\\[8pt]
 &\leq C_4^{'}+g'(t) -a(\sum\dfrac{|u_k|^2}{\lambda_k}+\sum\dfrac{1}{\lambda_k}),\\[8pt]
 \end{array}$
 \end{flushleft}
 where $a:=\min \{2B-(A+BC_1), \dfrac{\eta}{2}\}$.
 Hence, at $(z_0,t_0)$\\
 \begin{equation}
 \label{blocki}
 \sum\dfrac{|u_k|^2}{\lambda_k}+\sum\dfrac{1}{\lambda_k}\leq \dfrac{1}{a}(C_4^{'}+g'(t)).
 \end{equation}
 Moreover, by Proposition \ref{t} and by \eqref{sup u},  there exists $C_4^{''}=C_4^{''}(m,C_1,C_2)$ such that
 \begin{equation}
 \label{tich}
 \lambda_1\lambda_2...\lambda_n=\det (u_{\alpha\bar{\beta}})=e^{\dot{u}-f(t,z,u)}\leq C_4^{''}.
 \end{equation}
 By \eqref{blocki} and \eqref{tich}, there exists $C_4^{'''}=C_4^{'''}(a,C_4^{'},C_4^{''})$ such that\\
 \begin{center}
 $\lambda_k=\prod\lambda_j\prod\limits_{l\neq k}\dfrac{1}{\lambda_l}\leq (C_4^{'''}+g'(t_0))^{n-1}\;\;$ for 
 $\;\;k=1,...,n.$\\[8pt]
 $|\nabla u|^2=\sum|u_k|^2\leq ((C_4^{'''}+g'(t_0))^n\;\;$ for $\;\; (z,t)=(z_0,t_0).$\\
 \end{center}
 Then
 \begin{flushleft}
 $\begin{array}{ll}
 \phi (z_0,t_0)&\leq n\log(C_4^{'''}+g'(t_0))+g(t_0)+\gamma(u(z_0,t_0))+\eta|z_0|^2\\[8pt]
 &\leq n\log(C_4^{'''}(t_0-\frac{2}{m})+n)+\gamma(u(z_0,t_0))+\eta|z_0|^2\\[8pt]
 &\leq \tilde{C}_4.\\[8pt]
 \end{array}$
 \end{flushleft}
 For $z\in\Omega, \frac{2}{m}<\epsilon<t<T'$, we have
 $$
 \log |\nabla u|^2\leq \tilde{C}_4-\gamma (u)-\eta|z|^2-g(t)\leq 2\log C_4,
 $$
 where $C_4>0$ depends on $\Omega,m, \epsilon, T, C_f,C_1,C_2,C_3$.
 \end{proof}
 \section{Higher order estimates}
 \label{order2}
 In this section, we prove that the second derivatives of $u$ are bounded on $\partial\Omega\times (\epsilon,T)$. Then we use the maximum principle to show that the Laplacian of 
 $u$ is bounded on $\Omega\times (\epsilon,T)$. For convenience, we denote
 $\underline{u}:=\underline{u}_m$, $M:=M_m$, where $\frac{1}{2m}<\epsilon\leq \frac{1}{2m-1}$
 and $u_m, M_m$ are defined as in \ref{subsolution}.\\
 \subsection{Localisation technique}
 \begin{flushleft}
 In order to show that the second derivatives of $u$ are bounded on $\partial\Omega\times (\epsilon,T)$, we use a barrier function. The key to the construction is the following:
 \end{flushleft}
 \begin{Lem} \label{guan 98}
 We set
 $$v=(u-\underline{u})+a(h-\underline{u})-Nd^2,$$
 where $d$ is the distance from $\partial\Omega$, $h$ is defined as in the proof of Proposition 
 \ref{gra 1}  and $a, N$ are positive constants to be determined.
 Let $\epsilon\in (0,T)$. Then there exist $a, N, \delta >0$ depending only on $\Omega,\epsilon, T, C_u, C_{\varphi}, C_f$ such that
 \begin{center}
 $L(v)\geq\frac{1}{4}(1+\sum u^{\alpha\bar{\alpha}})\;\;$ on $\;\; U_{\delta}\times (\epsilon,T)$,\\[10pt]
 $v\geq 0\;\;$ on  $\;\; U_{\delta}\times (\epsilon,T)$,\\[6pt]
 \end{center}
 where $U_{\delta}=\{ z\in\Omega : d(z)\leq \delta\}$ .
 \end{Lem}
 \begin{proof}
 The elliptic version of this lemma was proved by \cite{Gua98} (page 5-7). The same 
 arguments can be applied for the parabolic case. For the reader's convenience, we
 recall the arguments here.\\
 We have
 $$
 L(v)=\dot{v}-n+\sum u^{\alpha\bar{\beta}}\underline{u}_{\alpha\bar{\beta}}-
 a\sum u^{\alpha\bar{\beta}} (h_{\alpha\bar{\beta}}-\underline{u}_{\alpha\bar{\beta}})
 +2N\sum u^{\alpha\bar{\beta}}(dd_{\alpha\bar{\beta}}+d_{\alpha}d_{\bar{\beta}})
 -f_u(t,z,u)v.
 $$
 Fix $\tilde{\delta}>0$ satisfying $d\in C^{\infty}(U_{\tilde{\delta}})$. Assume that $0<a<1$ and 
 $0<\delta<\tilde{\delta}$ and $0<N< \frac{1}{\delta}$. Then there exists $C_5>0$ depending on 
 $\Omega,\tilde{\delta}, \epsilon,T, C_{\varphi}, C_f, M, C_1, C_2$ such that
 \begin{center}
 $\dot{v}-n-f_u(t,z,u)v\geq -C_5$,\\
 $-a\sum u^{\alpha\bar{\beta}} (h_{\alpha\bar{\beta}}-\underline{u}_{\alpha\bar{\beta}})\geq 
 -C_5a\sum u^{\alpha\bar{\alpha}}$,\\
 $2Nd\sum u^{\alpha\bar{\beta}}d_{\alpha\bar{\beta}}\geq -C_5N\delta\sum
  u^{\alpha\bar{\alpha}},$
 \end{center}
 where $(z,t)\in U_{\delta}\times (\epsilon,T)$.
 
 Then
 $$
 L(v)\geq \sum u^{\alpha\bar{\beta}}\underline{u}_{\alpha\bar{\beta}}-C_5-C_5(a+N\delta)
 \sum u^{\alpha\bar{\alpha}}+2N\sum u^{\alpha\bar{\beta}}d_{\alpha}d_{\bar{\beta}},
 $$
 where $(z,t)\in U_{\delta}\times (\epsilon,T)$.
 
 When $a+N\delta\leq \frac{1}{4C_5}$, we obtain
 $$
 L(v)\geq \dfrac{3}{4}\sum u^{\alpha\bar{\alpha}}-C_5+2N
 \sum u^{\alpha\bar{\beta}}d_{\alpha}d_{\bar{\beta}},
  $$
  where $(z,t)\in U_{\delta}\times (\epsilon,T)$.
  
 Let $\lambda_1\leq \lambda_2\leq ...\leq\lambda_n$ be the eigenvalues of $\{
  u_{\alpha\bar{\beta}}\}$.   We have
  \begin{center}
  $\sum u^{\alpha\bar{\beta}}d_{\alpha}d_{\bar{\beta}}\geq \lambda_n^{-1}\sum d_{\alpha}d_{\bar{\alpha}}\geq \dfrac{\lambda_n^{-1}}{2}\;\;$ on $\;\; U_{\delta}\times (\epsilon,T)$.
  \end{center}
  By the inequality for arithmetic and geometric means
  \begin{center}
$\dfrac{1}{4}\sum u^{\alpha\bar{\alpha}}+N\lambda_n^{-1}
\geq n (\dfrac{1}{4})^{(n-1)/n}N^{1/n}(\lambda_1...\lambda_n)^{-1/n}\geq C_6N^{1/n}$,
\end{center} 
where $C_6>0$ depends on $\epsilon,T,C_f, C_1,C_2$.\\
  When $N>(\frac{C_5+1}{C_6})^n$, we have
  \begin{center}
  $L(v)\geq \dfrac{1}{2}(2+\sum u^{\alpha\bar{\alpha}})$.
  \end{center}
  Next, since $\Delta\underline{u}\geq n$, there exists $C_7>0$ depending only on $\Omega$ such that
  \begin{center}
  $(h-\underline{u})\geq C_7 d\;\;$ on $\;\; \Omega\times (\epsilon, T)$.
  \end{center}
  Fix $0<a, \delta <1$, $N>0$ so that
  \begin{itemize}
   \item $N>\left(\dfrac{C_5+1}{C_6}\right)^n$;\\
  \item $a \leq \dfrac{1}{8C_5}$;\\
  \item $0<\delta<\tilde{\delta}$;\\
  \item $\min\{aC_7,a\}\geq N\delta$.
  \end{itemize}
  We obtain
  \begin{center}
 $L(v)\geq\frac{1}{4}(1+\sum u^{\alpha\bar{\alpha}})\;\;$ on $\;\; U_{\delta}\times (\epsilon,T)$,
 \\[10pt]
 $v\geq 0\;\;$ on  $\;\; U_{\delta}\times (\epsilon,T)$.\\[6pt]
 \end{center}
 \end{proof}
 
 \subsection{$\mathcal{C}^2$-a priori estimates on the boundary}
 
 \begin{Lem} \label{bou 11}
 Let $\epsilon\in (0,T)$. Then there exists $c_{\epsilon} >0$ depending only on $\Omega, \epsilon, T, C_u, C_{\varphi},C_f$ such that
 \begin{center}
 $(dd^cu)|_{T^h_{\partial\Omega}}\geq c_{\epsilon}(dd^c|z|^2 )|_{T^h_{\partial\Omega}}$,
 \end{center}
 where $T^h_{\partial\Omega}$ is the holomorphic tangent bundle of $\partial\Omega$.
 \end{Lem}
 We refer the reader to \cite[pp. 221--223]{CKNS85} or \cite[p. 268--271]{Bou11}  for related
 results in the elliptic case.
 
 \begin{proof} 
 Fix $p\in\partial\Omega$ . By an affine change of coordinates, we can assume that $p=0$ and there
 exists a neighbourhood $U$ of $p$ such that
 $$\Omega\cap U=\{z\in U: x_n>Re(\sum\limits_{1\leq j\leq k\leq n}a_{j\bar{k}}z_j\bar{z}_k+\sum\limits_{1\leq j\leq k \leq n}a_{jk}z_jz_k)+O(|z|^3)\},$$
where $a_{j\bar{k}}, a_{jk}\in\C$ with $a_{1\bar{1}}>0$.

By a holomorphic change of coordinates, we can assume that 
\begin{equation}\label{nearpholo.bou11}
\Omega\cap U=\{z\in U: x_n>Re(\sum\limits_{1\leq j\leq k\leq n}a_{j\bar{k}}z_j\bar{z}_k)+O(|z|^3)\},
\end{equation}
 where $a_{j\bar{k}}$ with $a_{1\bar{1}}>0$.
 
  We need to show that
 $$
 u_{1\bar{1}}(p,t)\geq C_{\epsilon},
 $$
 where $t\in (\epsilon,T)$ and $C_{\epsilon}>0$ depends on $\Omega, \epsilon, T, C_u, C_{\varphi},C_f$.
 
 \textit{Step 1: Choice of a K\"ahler potential.}\\
  We construct a function $\tau\in C^{\infty}(\Omega_r\times (\epsilon , T))$ 
  depending on $\U ,\epsilon , T, \Omega$ so that $dd^c\tau =dd^c\U$ and $\tau 
  (p,t)=0$ and
  $$
  \tau |_{(\partial\Omega\cap B_r)\times (\epsilon,T)}=
  Re\left(\sum\limits_{j=2}^n c_jz_1\bar{z}_j\right) + 
  O\left( |z_2|^2+...+|z_n|^2\right),
  $$
  where $r>0$, $B_r=B_r(p)$, $\Omega_r=\Omega\cap B_r$ and $c_j\in C^{\infty}([\epsilon, T), \C)$.\\
Indeed, by Taylor's formula,
\begin{flushleft}
$\begin{array}{ll}
\U(z,t)-\U(p,t)&=Re(\sum\limits_{j=1}^nb_jz_j)+Re(\sum\limits_{j=2}^nb_{1\bar{j}}z_1\bar{z}_j)
+b_{1\bar{1}}|z_1|^2+Re(\sum\limits_{j=1}^nb_{1j}z_1z_j)\\
&+O(|z_2|^2+...+|z_n|^2)+O(|z|^3),
\end{array}$
\end{flushleft}
where $b_j,b_{1j},b_{1\bar{j}}\in\C^{\infty}([\epsilon,T),\C )$,
$b_{1\bar{1}}=\U_{1\bar{1}}(p,t)>0$.

Furthermore, near $p$ on $\partial\Omega$, we have by \eqref{nearpholo.bou11}
\begin{equation}\label{xn}
x_n=Re(\sum\limits_{j=2}^na_{1\bar{j}}z_1\bar{z}_j)
+a_{1\bar{1}}|z_1|^2+O(|z_2|^2+...+|z_n|^2)+O(|z|^3),
\end{equation}
where $a_{1\bar{j}}\in\C$ with $a_{1\bar{1}}>0$.\\
Define
$$
\tau(z,t)=\U (z,t)-\U (p,t)-Re(\sum\limits_{j=1}^nb_jz_j)
-\dfrac{b_{1\bar{1}}}{a_{1\bar{1}}}x_n
-Re(\sum\limits_{j=1}^nb_{1j}z_1z_j);
$$
then   $dd^c\tau =dd^c\U$ and $\tau (p,t)=0$ and
  $$
  \tau |_{(\partial\Omega\cap B_r)\times (\epsilon,T)}=Re\left(\sum\limits_{j=2}^n
   c_jz_1\bar{z}_j\right) + 
  O(|z_2|^2+...+|z_n|^2)+ \{\mbox{terms of order}\geq 3\}.
  $$
  Moreover, for $z\in\partial\Omega$, we have
  \begin{itemize}
  \item  For $j=2,...,n$\\
  \begin{equation}
  \label{zj2z1}
 |z_j|^2|z_1|=O(|z_2|^2+...+|z_n|^2);
 \end{equation}
  \item By \eqref{xn}
  \begin{flushleft}
  $\begin{array}{ll}
  |z_1|^4
  &=O(x_n^2)+O(\sum\limits_{j=2}^n|z_1|^2|z_j|^2)+O(|z|^6)+O((\sum\limits_{j=2}^n|z_j|^2)^2)\\
  &=O(|z_2|^2+...+|z_n|^2)+O(|z|^6);
  \end{array}$
  \end{flushleft}
  then
 \begin{equation}
 \label{z4}
 |z|^4=O(|z_2|^2+...+|z_n|^2);
 \end{equation}
  \item  For $j=2,...,n$
  \begin{equation}\label{z12zj}
  |z_1|^2|z_j|=O(|z_1|^4)+ O( |z_j|^2)=O(|z_2|^2+...+|z_n|^2).
  \end{equation}
  \end{itemize}
  Hence
  \begin{flushleft}
  $\begin{array}{ll}
  \tau |_{(\partial\Omega\cap B_r)\times (\epsilon,T)}
  &=Re (\sum\limits_{j=2}^n c_jz_1\bar{z}_j)
  +\sum\tilde{a}_jx_1^jy_1^{3-j}
  +O(|z_2|^2+...+|z_n|^2)\\
  &=Re(\sum\limits_{j=2}^n
   c_jz_1\bar{z}_j)  
  +Re(a_1z_1^3)+Re(a_2z_1|z_1|^2)+O(|z_2|^2+...+|z_n|^2),
   \end{array}$
  \end{flushleft}
  where $a_1,a_2\in C^{\infty}([\epsilon,T),\C)$ .\\
  Next, by \eqref{xn}, \eqref{zj2z1}, \eqref{z12zj}, for $z\in\partial\Omega$, we have
  \begin{flushleft}
  $\begin{array}{ll}
  Re(a_2z_1|z_1|^2)
  &=Re(\dfrac{a_2}{a_{1\bar{1}}}z_1x_n)+ 
  O(|z_2|^2+...+|z_n|^2)\\
  &=Re(c_0z_1\bar{z}_n)+Re(c_0z_1z_n)+
   O(|z_2|^2+...+|z_n|^2).
  \end{array}$
  \end{flushleft}
  Replacing the term  $c_n$ by $c_n-c_0$, we obtain
  \begin{flushleft}
  $\begin{array}{ll}
  \tau |_{(\partial\Omega\cap B_r)\times (\epsilon,T)}&=Re\left(\sum\limits_{j=2}^n
   c_jz_1\bar{z}_j\right) +Re(a_1z_1^3)+Re(c_0z_1z_n)+O(|z_2|^2+...+|z_n|^2).
  \end{array}$
  \end{flushleft}
  Replacing $\tau$ by $\tau+Re(a_1z_1^3)+Re(c_0z_1z_n)$, we obtain\\
  $$\tau |_{(\partial\Omega\cap B_r)\times (\epsilon,T)}=Re\left(\sum\limits_{j=2}^n
   c_jz_1\bar{z}_j\right)  
  +O(|z_2|^2+...+|z_n|^2).
  $$
 Therefore,
  \begin{equation}
  \label{taumaj}
  \tau |_{(\partial\Omega\cap B_r)\times (\epsilon,T)}\leq 
  Re\left(\sum\limits_{j=2}^n c_jz_1\bar{z}_j\right)  
  +a_3(|z_2|^2+...+|z_n|^2), \quad
  \sup\sum\limits_{j=2}^n|c_j|\leq a_4,
  \end{equation}
  where $a_3,a_4>0$ depend on $\Omega, \epsilon, T, M, C_{\varphi}$.\\
  The conditions $dd^c\tau =dd^c\U$ and $\tau (p,t)=0$ are still satisfied.\\
 \textit{Step 2: Choice of a barrier function.}\\
Recall that $\Omega_r=\Omega \cap B_r$. We construct a function 
\begin{equation}
\label{bdef}
b(z,t)=-\epsilon_1 x_n +\epsilon_2 |z|^2+\dfrac{1}{2\mu}\sum\limits_{j=2}^{n}
 |c_jz_1+\mu z_j|^2
\end{equation}
 such that $b\geq \tau + u-\U\;$ on $\; \Omega_r\times (\epsilon , T)$,
 where $r>0$ depends only on $\Omega$ and 
 $\epsilon_1, \epsilon_2, \mu>0$ depend on $\Omega, \epsilon, T, M, C_{\varphi},C_f$.\\
 Note that
 \begin{center}
 $|z_1|^2\leq\dfrac{1}{a_{1\bar{1}}} (x_n-Re(\sum\limits_{j=2}^na_{1\bar{j}}z_1\bar{z}_j))
+O(|z_2|^2+...+|z_n|^2)+O(|z|^3)$ on $\Omega$.
 \end{center}
 Since for $r_0$ small enough and $z\in \Omega_{r_0}$,we have $z\rightarrow 0$ as 
 $|z_2|^2+...+|z_n|^2\rightarrow 0$,
 if we fix $r>0$ small enough, then there exists $r_1>0$ such that
 \begin{center}
 $|z_2|^2+...+|z_n|^2\geq r_1\;\;$ for $\;\; z\in\partial B_r\cap\Omega$.
 \end{center}
 Assume that $0<\epsilon_1,\epsilon_2<1$. 
 Then there exists $\mu_1>0$ depending on $\Omega, M, C_{\varphi},
 C_1, a_3,a_4,r_1$ such that the function $b$ in \eqref{bdef} verifies
 \begin{flushleft}
 $\begin{array}{ll}
 b|_{(\partial B_r(p)\cap\Omega)\times [\epsilon, T)}
 &\geq \dfrac{\mu r_1}{2} + 
 Re(\sum\limits_{j=2}^n c_jz_1\bar{z}_j)-\epsilon_1x_n+\epsilon_2|z|^2\\[8pt]
 &\geq \dfrac{\mu_1 r_1}{2} + 
 Re(\sum\limits_{j=2}^n c_jz_1\bar{z}_j)-\epsilon_1x_n+\epsilon_2|z|^2\\[8pt]
 &\geq (\tau+u-\U)|_{(\partial B_r(p)\cap \Omega )\times [\epsilon,T)}\\
 \end{array}$
 \end{flushleft}
 when $\mu\geq\mu_1$.\\
 There exists $r_2>0$ such that, when $z \in \partial\Omega$, \\
 $$
 x_n= Re(\sum\limits_{j=1}^na_{1\bar{j}}z_1\bar{z}_j)+O(|z_2|^2+...+|z_n|^2)+O(|z|^3)\leq r_2|z|^2.
 $$
 Assume that $0<r_2\epsilon_1<\epsilon_2$. For $\mu\geq 2a_3$, 
 by \eqref{taumaj}, we have 
 \begin{flushleft}
 $\begin{array}{ll}
 b|_{(\partial\Omega\cap B_r(p))\times [\epsilon,T)}
 &\geq \dfrac{1}{2\mu}\sum\limits_{j=2}^n|c_jz_1+\mu z_j|^2\\[8pt]
 &\geq Re(\sum\limits_{j=2}^n c_jz_1\bar{z}_j)+\dfrac{\mu}{2}(|z_2|^2+...+|z_n|^2)\\[8pt]
 &\geq \tau|_{(\partial\Omega\cap B_r(p))\times [\epsilon,T)}\\[8pt]
 &\geq (\tau+u-\U)|_{(\partial\Omega\cap B_r(p))\times [\epsilon,T)} .
 \end{array}$
 \end{flushleft}
 Fix $\mu\geq \max(\mu_1, 2 a_3)$, we get
 $$
 b|_{\partial_P(\Omega_r\times [\epsilon,T))}\geq 
 (\tau+u-\U)|_{\partial \Omega_r\times [\epsilon,T)}.
 $$
 Next, by Proposition \ref{t} ,there exists $r_3>0$ such that
 \begin{center}
 $(dd^c(\tau-u-\U))^n=(dd^cu)^n=e^{\dot{u}-f(t,z,u)}\geq r_3\;\;$ on $\;\; \Omega_r\times[\epsilon,T).$
 \end{center}
 On the other hand
 $$
 (dd^c(\sum\limits_{j=2}^n|c_jz_1+\mu z_j|^2))^n=0,
 $$
 so $(dd^cb)^n=O(\epsilon_2)$ on $\Omega_r\times [\epsilon,T)$.\\
  Hence, there exists $\epsilon_2 >0$ depending on $\mu, \Omega, a_4,r_3$ such that
  \begin{center}
  $(dd^cb)^n\leq (dd^c(\tau+u-\U))^n$ on $\Omega_r\times [\epsilon, T)$.
  \end{center}
  When $b|_{\partial\Omega_r\times [\epsilon,T)}\geq (\tau+u-\U)|_{\partial\Omega_r\times [\epsilon,T)}$
  and $(dd^cb)^n\leq (dd^c(\tau +u-\U))^n$ on $\Omega_r\times [\epsilon, T)$, it follows from 
  the comparison theorem (for the bounded plurisubharmonic functions) that
  \begin{center}
  $b\geq (\tau+u-\U)\;\;$ on $\;\;\Omega_r\times [\epsilon ,T)$.\\
  \end{center}
  
 \textit{Step 3: Conclusion.}\\
 We have, since $b(p,t)= \tau(p,t)+ u(p,t)-\U (p,t)=0$, 
 $$
 -\epsilon_1=b_{x_n}(p,t)\geq \tau_{x_n}(p,t) + (u-\U)_{x_n}(p,t).
 $$
 Then, since $(u-\underline{u})|_{\partial \Omega\times (\epsilon, T)} \equiv 0$, 
 $$
 (u-\underline{u})_{1\bar 1} (p,t )= -(u-\underline{u})_{x_n} (p,t ) \rho_{1\bar 1} (p ),
 $$
 and by the explicit choice of $\tau$,  
  $-\tau_{x_n}(p,t) \rho_{1\bar 1} (p )=\tau_{1\bar{1}}(p,t)$, so
$$
u_{1\bar{1}}(p,t)= (\tau_{1\bar{1}}+u_{1\bar{1}}-\U_{1\bar{1}})(p,t)=
-\left( \tau_{x_n}(p,t)+(u-\U)_{x_n}(p,t)\right)\rho_{1\bar{1}}(p)
\geq \epsilon_1 \rho_{1\bar{1}}( p ). 
$$
 \end{proof}
 \begin{Prop} \label{second 1}
 There exists $D_1=D_1 (\Omega, \epsilon, T, C_u, C_{\varphi}, C_f)$ such that\\
 \begin{center}
 $|D^2u|\leq D_1\;\;$ on $\;\; \partial\Omega\times (\epsilon , T).$ 
 \end{center}
 \end{Prop}
 \begin{proof}
 Fix $p\in\partial\Omega$. We can choose complex coordinates $(z_j)_{1\leq j\leq n}$ so that $p=0$ and the positive $x_n$ axis is the interior normal direction of $\partial\Omega$ at $p$. We set for convenience
 \begin{center}
 $s_1=y_1,s_2=x_1,...,s_{2n-1}=y_n, s_{2n}=x_n, s^{'}=(s_1,...,s_{2n-1}).$
 \end{center}
 We also assume that near $p$, $\partial\Omega$ is represented as a graph
 \begin{center}
 $x_n=P(s^{'})=\sum\limits_{j,k<2n} P_{jk}s_js_k+O(|s^{'}|^3).$
 \end{center}
 \textit{Step 1: Bounding the tangent-tangent derivatives.}\\
 Since $(u-\underline{u})(s{'},P(s{'}),t)=0$, we have for $j,k<2n,\; 0<t<T\;$:
 \begin{center}
 $(u-\underline{u})_{s_js_k}(p,t)=-(u-\underline{u})_{x_n}(p,t)P_{jk}$.
 \end{center}
 By Proposition \ref{gra 1}, we obtain
 $$|u_{s_js_k}(p,t)|\leq D^{'}_1,$$
 where $D^{'}_1>0$ depends only on $\Omega,C_{\varphi}, M$.\\
 \textit{Step 2: Bounding the normal-tangent derivatives.}\\
 Define\\
 $$T_j=\frac{\partial}{\partial s_j}+P_{s_j}\frac{\partial}{\partial x_n}.$$
 Again, denote $\Omega_{\delta}=B_{\delta}(p)\cap \Omega$. With $v$ as
 in Lemma \ref{bou 11}, we construct the functions
 \begin{center}
 $\psi_{\pm}=Av+B|z|^2-(t-\frac{\epsilon}{2})(u_{y_n}-\underline{u}_{y_n})^2\pm (t-\frac{\epsilon}{2}) T_j(u-\underline{u})$,
 \end{center}
 such that
 \begin{center}
 $L(\psi_{\pm})\geq 0\;\;$ on $\;\; \Omega_{\delta}\times (\frac{\epsilon}{2},T),$\\
 $\psi_{\pm}\geq 0\;\;$ on $\;\; \Omega_{\delta}\times (\frac{\epsilon}{2},T),$\\
 \end{center}
 where $A,B>0$ depend on $\Omega, C_{\varphi},C_f,\epsilon, T,  M$.\\
 We compute
 \begin{flushleft}
$\begin{array}{ll}
L(-(u_{y_n}-\U_{y_n})^2)
&=-2(u_{y_n}-\U_{y_n})L(u_{y_n}-\U_{y_n})-f_u(t,z,u)(u_{y_n}-\U_{y_n})^2\\
&+2\sum u^{\alpha\bar{\beta}} (u_{y_n}-\U_{y_n})_{\alpha}(u_{y_n}-\U_{y_n})_{\bar{\beta}}\\
\end{array}$
\end{flushleft}
and
\begin{flushleft}
$\begin{array}{ll}
L(\pm T_j(u-\U))
&=\pm L(u_{s_j}-\U_{s_j})\pm P_{s_j}L(u_{x_n}-\U_{x_n})\\
&\mp (u_{x_n}-\U_{x_n})\sum u^{\alpha\bar{\beta}} (P_{s_j})_{\alpha\bar{\beta}}\\
&\mp \sum u^{\alpha\bar{\beta}}\left((u_{x_n}-\U_{x_n})_{\alpha}
(P_{s_j})_{\bar{\beta}}+(u_{x_n}-\U_{x_n})_{\bar{\beta}}(P_{s_j})_{\alpha}\right).\\
\end{array}$
\end{flushleft}
 By equation \eqref{eqmain}, for $k=1,2,..., 2n$
 \begin{center}
$L(u_{s_k}-\underline{u}_{s_k})= f_{s_k}(t,z,u)-\dot{\underline{u}}_{s_k}
+\sum u^{\alpha\bar{\beta}} (\underline{u}_{s_k})_{\alpha\bar{\beta}}
+\underline{u}_{s_k}f_u(t,z,u).$
\end{center}
Hence
\begin{flushleft}
$\begin{array}{ll}
&L(-(u_{y_n}-\U_{y_n})^2\pm T_j(u-\U))\\[8pt] 
&\geq -C_8(1+\sum u^{\alpha\bar{\alpha}})
+2\sum u^{\alpha\bar{\beta}} (u_{y_n}-\U_{y_n})_{\alpha}(u_{y_n}-\U_{y_n})_{\bar{\beta}}\\[8pt]
&\mp \sum u^{\alpha\bar{\beta}}\left((u_{x_n}-\U_{x_n})_{\alpha}
(P_{s_j})_{\bar{\beta}}+(u_{x_n}-\U_{x_n})_{\bar{\beta}}(P_{s_j})_{\alpha}\right),\\[8pt]
\end{array}$
\end{flushleft}
where $C_8>0$ depend on $\epsilon , C_1, C_2,C_3,M, C_{\varphi}, C_f, \rho,P$.\\
On the other hand
$$
\sum\limits_{\alpha=1}^n u^{\alpha\bar{\beta}}u_{x_n\alpha}=
2\delta_{\beta n}-i\sum\limits_{\alpha=1}^n u^{\alpha\bar{\beta}}u_{y_n\alpha},
$$
$$
\sum\limits_{\beta=1}^n u^{\alpha\bar{\beta}}u_{x_n\bar{\beta}}=
2\delta_{\alpha n}+i\sum\limits_{\beta=1}^n u^{\alpha\bar{\beta}}u_{y_n\bar{\beta}}.
$$
Then
\begin{flushleft}
$\begin{array}{ll}
&L(-(u_{y_n}-\U_{y_n})^2\pm T_j(u-\U))\\[8pt] 
&\geq -C_9(1+\sum u^{\alpha\bar{\alpha}})
+2\sum u^{\alpha\bar{\beta}} (u_{y_n}-\U_{y_n})_{\alpha}(u_{y_n}-\U_{y_n})_{\bar{\beta}}\\[8pt]
&\mp \sum u^{\alpha\bar{\beta}}\left((u_{y_n}-\U_{y_n})_{\alpha}
(-iP_{s_j})_{\bar{\beta}}+(u_{y_n}-\U_{y_n})_{\bar{\beta}}(iP_{s_j})_{\alpha}\right)\\[8pt]
\end{array}$
\end{flushleft}
where $C_9>0$ depend on $\epsilon , C_1, C_2,C_3,M, C_{\varphi}, C_f, \rho,P$.\\
By the Cauchy-Schwarz inequality,\\
$2\sum u^{\alpha\bar{\beta}} (u_{y_n}-\U_{y_n})_{\alpha}(u_{y_n}-\U_{y_n})_{\bar{\beta}}
+\dfrac{1}{2}\sum u^{\alpha\bar{\beta}}(iP_{s_j})_{\alpha}(-iP_{s_j})_{\bar{\beta}}$\\
$\geq \pm \sum u^{\alpha\bar{\beta}}\left((u_{y_n}-\U_{y_n})_{\alpha}
(-iP_{s_j})_{\bar{\beta}}+(u_{y_n}-\U_{y_n})_{\bar{\beta}}(iP_{s_j})_{\alpha}\right)$.\\
Then 
$$
L(-(u_{y_n}-\U_{y_n})^2\pm T_j(u-\U))\geq -C_{10}(1+\sum u^{\alpha\bar{\alpha}}),
$$
where $C_{10}>0$ depends on $\Omega, C_{\varphi},C_f,\epsilon, T, M$.\\
Hence, by Lemma \ref{bou 11}, we can choose $A,B>0$ independent of $u$ so that
\begin{center}
$L(\psi_{\pm})\geq 0\;\;$ on $\;\;\Omega_{\delta}\times (\epsilon, T),$\\
$\psi_{\pm} \geq 0\;\;$ on $\;\;\partial_P(\Omega_{\delta}\times (\epsilon, T)).$
\end{center}
 By the maximum principle, we obtain
 $\psi_{\pm}\geq 0$ on $\Omega_{\delta}\times (\frac{\epsilon}{2}, T)$.
 
 Note that $\psi_{\pm} (p,t)=0\;$ for $\; t\in (\frac{\epsilon}{2},T)$.
 
 Hence, 
 $$
 \lim\limits_{x_n\searrow 0} \dfrac{\psi_{\pm} (p+(0,\dots,x_n),t)-\psi_{\pm} (p,t)}{x_n}\geq 0,
 $$
thus
 $$
 |u_{s_jx_n}(p,t)|\leq D^{''}_1,
 $$
  where $t\in (\epsilon,T)$ and $D{''}_1>0$ depend
  only on $\Omega, C_{\varphi},C_f,\epsilon, T, C_u$.\\ 
  \textit{Step 3:Bounding the normal-normal derivatives.}\\
  We have that
    $$
    \det (u_{\alpha\bar{\beta}})=e^{\dot{u}-f(t,z,u)}$$
  is bounded from above and below on $\partial\Omega\times (\epsilon,T)$. \\
  By step 1 and step 2,
  $|u_{z_n\bar{z}_n}\det (u_{\alpha\bar{\beta}})_{\alpha,\beta\leq n-1}|$
  is bounded on $\{ p\}\times (\epsilon,T)$.\\
  Hence, by Lemma \ref{bou 11} , we obtain\\
  $$
  |u_{z_n\bar{z}_n}(p,t)|\leq D^{'''}_1\,\, , \,\, t\in (\epsilon,T),
  $$
  where $D^{'''}_1$ depends on $\Omega, C_{\varphi},C_f,\epsilon, T,C_u$.\\
  Consequently\\
  $$|u_{x_nx_n}|\leq D^{''''}_1,$$
  where $D^{''''}_1$ depends on $\Omega, C_{\varphi},C_f,\epsilon, T, C_u$.\\
 \end{proof}
 \subsection{Interior estimate of the Laplacian}
 \begin{Prop}\label{second 2}
  There exists $D_2=D_2 (\Omega, \epsilon, T, C_{\varphi}, C_f, C_u)$ such that\\
 \begin{center}
 $\Delta u\leq D_2\;\;$ on $\;\; \Omega\times (\epsilon , T).$ 
 \end{center}
 \end{Prop}
 \begin{proof}
 We set 
 $$\phi=(t-\epsilon)\log\Delta u+A_1 |z|^2-A_2t,$$
 where $A_1,A_2>0$ will be specified later.\\
 We have
 \begin{flushleft}
 $\begin{array}{ll}
 L(\phi)&=  \log\Delta u +(t-\epsilon)\dfrac{\Delta\dot{u}}{\Delta u}-A_2-(t-\epsilon)\sum u^{\alpha\bar{\beta}}(\log \Delta u)_{\alpha\bar{\beta}}\\[8pt]
  & -A_1\sum u^{\alpha\bar{\alpha}} -\phi f_u(t,z,u).
  \end{array}$
  \end{flushleft}
 By Theorem \ref{lap 1},
 $$
 \log\Delta u\leq \log n + \log\det (u_{\alpha\bar{\beta}})+(n-1)\log (\sum u^{\alpha\bar{\alpha}}).
 $$
 By Theorem \ref{lap 2},
 \begin{flushleft}
 $\begin{array}{ll}
 \dfrac{\Delta\dot{u}}{\Delta u}-\sum u^{\alpha\bar{\beta}}(\log \Delta u)_{\alpha\bar{\beta}}&\leq\dfrac{\Delta\dot{u}}{\Delta u}-\dfrac{\Delta\log\det (u_{\alpha\bar{\beta}})}{\Delta u}\\[8pt]
 & =\dfrac{\Delta f(t,z,u)}{\Delta u}\\[8pt]
 &=\dfrac{\Delta_z f(t,z,u)}{\Delta u} +f_u(t,z,u) +\sum \dfrac{f_{u s_j}(t,z,u)u_{s_j}}{\Delta u}\\
 &+\sum \dfrac{f_{uu}(t,z,u) u_{s_j}^2}{\Delta u}.\\
 \end{array}$
 \end{flushleft}
Hence, there exist $A_1, A_2>0$ depending on $\Omega, \epsilon, T, C_{\varphi}, C_f, C_u$ such that\\
\begin{center}
$L(\phi)\leq 0$ on $\Omega\times (\epsilon, T)$.
\end{center}
Thus, by the maximum principle and Proposition \ref{second 1},\\
\begin{center}
$(t-\epsilon)\log\Delta u\leq D_2^{'}\;\;$ on $\;\;\Omega\times (\epsilon , T),$
\end{center}
where $D_2^{'}$ depends on $\Omega, \epsilon, T, C_{\varphi}, C_f, C_u$.\\
Therefore,\\
\begin{center}
$\Delta u\leq e^{D_2^{'}/\epsilon}\;\;$ on $\;\;\Omega\times (2\epsilon , T).$
\end{center}
 \end{proof}
 \section{$C^{2,\alpha}$ estimate up to the boundary for the parabolic equation}
 \label{C2a}
 \subsection{Parabolic H\"older spaces}
 \begin{flushleft}
 The reader can find more complete notations in  \cite[Chapter 4]{Lieb96} or
  \cite[Chapter 8]{Kryl96}.\\
 \end{flushleft}
 In $\R^N\times\R$ we define the parabolic distance between the points $X_1=(x_1,t_1)$, 
 $X_2=(x_2,t_2)$ as\\
 $$d(X_1,X_2)=|x_1-x_2|+|t_1-t_2|^{1/2}.$$
 Let $0<\alpha<1$. Let $u$ be a function defined in a domain $Q\subset\R^N\times\R$. We say 
 that $u$ is uniformly H\"older continuous in $Q$ with exponent $\alpha$, 
 or $u\in C^{\alpha}(Q)$, if and only if
 \begin{center}
 $[u]_{\alpha;Q}=\sup\limits_{X_j\in Q, X_1\neq X_2 }\dfrac{|u(X_1)-u(X_2)|}{d^{\alpha}
 (X_1,X_2)}<\infty$.
 \end{center}
 Let $0<\beta<2$. We denote
 \begin{center}
 $\langle u \rangle_{\beta;Q}=\sup\limits_{(x,t_1)\neq(x,t_2)\in Q }
 \dfrac{|u(x,t_1)-u(x,t_2)|}
 {|t_1-t_2|^{\beta /2}}$.
 \end{center}
 We say that $u$ is uniformly H\"older continuous in $Q$ with exponent $k+\alpha$, or
 $u\in C^{k,\alpha}(Q)$ if the derivatives $D_x^jD_t^lu$ exist for 
 $|j|+2l\leq k$ and the norm
 \begin{center}
 $\|u\|_{C^{k,\alpha}(Q)}=\sum\limits_{|j|+2l\leq k}\sup\limits_Q |D_x^jD_t^lu|
 +\sum\limits_{|j|+2l= k} [D_x^jD_t^lu]_{\alpha;Q}
 +\sum\limits_{|j|+2l= k-1}\langle D_x^jD_t^lu\rangle_{\alpha+1;Q}$
 \end{center}
 is finite.\\
The norm $\|.\|_{C^{k,\alpha}(Q)}$ makes $ C^{k,\alpha}(Q)$ a Banach space. If we define
the similar notions for $\bar{Q}$, then $ C^{k,\alpha}(Q)=C^{k,\alpha}(\bar{Q})$.
 \subsection{$\bf{C^{2,\alpha}}$ estimate up to the boundary }\label{C2a.subsec}
 \begin{flushleft}
 Let $\Omega$ be a bounded smooth domain of $\R^N$.
 We consider the equation
 \end{flushleft}
 \begin{equation}
 \label{intro.para}
 \dot{u}=F(D^2u)+f(t,x,u) \; \mbox{in}\; \Omega\times (0,\tilde{T}),
 \end{equation}
 where $\tilde{T}>0$, $f$ is a smooth function defined on $[0,\tilde{T})
 \times\bar{\Omega}\times\R$
 and $F$ is a smooth concave function defined on the set of all real 
 $N\times N$ matrices.
 In addition, we assume that there exist $0<\lambda<\Lambda<\infty$ such that\\
 \begin{equation}\label{F}
 \lambda \mbox{ \rm tr} \eta\leq F(r+\eta)-F(r)\leq \Lambda \mbox{ \rm tr} \eta
 \end{equation}
 for any symmetric matrix $r$, any positive definite matrix $\eta$.\\
 We will establish $C^{2,\alpha}$ estimates for the solution of \eqref{intro.para} on 
 $\bar{\Omega}\times (\epsilon, T)$ for any $0<\epsilon<T<\tilde{T}$ without 
 $C^{2,\alpha}$ conditions on $\Omega\times \{ 0\}$. 
 The main result of this section is the following:
\begin{The}
\label{main.C2alpha}
 Let $F$ be concave and smooth satisfying \eqref{F}. Let $f$ be a smooth function in
$[0,\tilde{T})\times\bar{\Omega}\times\R$ and $\varphi$ be a smooth function in 
 $\bar{\Omega}\times [0,\tilde{T})$. Assume that 
 $u\in C^{2;1}(\bar{\Omega}\times [0,\tilde{T}))\cap C^{\infty}(\Omega\times (0,\tilde{T}))$ 
 is a solution of
\begin{equation}\label{parabollic}
\begin{cases}
\begin{array}{ll}
\dot{u}=F(D^2u)+f(t,x,u) \; &\mbox{in}\; \Omega\times (0,\tilde{T}),\\
u=\varphi&\mbox{on}\; \partial\Omega\times(0,\tilde{T}),
\end{array}
\end{cases}
\end{equation}
and that
$$|u|+|\dot{u}|+|\nabla u|+|D^2u|\leq C,$$
then 
$u\in C^{2,\alpha}(\bar{\Omega}\times (0,\tilde{T}))$
 satisfies
\begin{equation}\label{C2alpha}
\|u\|_{C^{2,\alpha}(\Omega\times (\epsilon,T))}\leq 
C_{\epsilon,T}\;\;\;\forall 0<\epsilon<T<\tilde{T},
\end{equation}
where $0<\alpha<1$, $C_{\epsilon,T}>0$ depend on 
$\lambda,\Lambda,\Omega,C,\epsilon,T$ and the upper bound of 
$\|\varphi\|_{C^4}+\|F\|_{C^1}+\|f\|_{C^2}$.
\end{The}
\begin{Rem}
In the theorem above, we denote
$$\|\varphi\|_{C^k(\Omega\times (0,\tilde{T}))}
=\sum\limits_{|j|+2l\leq k}\sup\limits_{\Omega\times (0,\tilde{T})}
|D^j_xD^l_t\varphi|,
$$
$$
\|F\|_{C^k(Mat(N\times N,\R))}=\sum\limits_{|j|\leq k}\sup |D^jF|,
$$
$$\|f\|_{C^k((0,\tilde{T})\times\Omega\times\R))}
=\sum\limits_{j_1+|j_2|+j_3\leq k}\sup |D_t^{j_1}D_x^{j_2}D_u^{j_3}f|.
$$
\end{Rem}
In order to prove Theorem \ref{main.C2alpha}, we use the technique of Caffarelli as in
\cite{CC95}. We need to prove a series of lemmas.
 \begin{Lem}\label{b-b}
 There exist  $0<\beta<1$ and $C_{\epsilon,T}>0$ depending on 
$\lambda,\Lambda,\Omega,C,\epsilon,T$ and the upper bound of 
$\|\varphi\|_{C^4}+\|F\|_{C^1}+\|f\|_{C^1}$ such that
 $$\dfrac{\|D^2u(x,t)-D^2u(x_0,t_0)\|}{(|x-x_0|+|t-t_0|^{1/2})^{\beta}}
 \leq C_{\epsilon,T},
 \;\;\;\forall x,x_0\in\partial\Omega;
 \forall t,t_0\in (\epsilon,T).$$
 \end{Lem}
 \begin{proof}
 Let $x_0\in\partial\Omega$. We consider a smooth diffeomorphism
 \begin{center}
 $\psi : U\cap\Omega\longrightarrow B^+_4:=\{y\in\R^N:|y|<4,y_N>0\}$\\
 $x\mapsto y=\psi(x)$
 \end{center}
 such that $\psi (x_0)=0$ and
 $$\psi (U\cap\partial\Omega)=\Gamma_4=\{ y=(y',y_N)\in \R^{N-1}\times\R:
 |y'|<4,y_N=0\},$$
 where $U$ is a neighborhood of $x_0$.\\
 We define
 \begin{center}
 $v(y,t)=u(\psi^{-1}(y),t)-\varphi(\psi^{-1}(y),t),$
 \end{center}
 where $y\in B^+_4\bigcup\Gamma_4$, $t\in (\epsilon,T)$. 
 Then $v|_{\Gamma_4\times (\epsilon,T)}=0$ and $v$ satisfies the equation
 \begin{equation}\label{lembb.G}
 \dot{v}=G(t,y,v,Dv,D^2v)
 \end{equation}
 where the upper bound of $\|G\|_{C^1}$ depends on $\|F\|_{C^1}$, $\|f\|_{C^1}$
 and $\psi$. Moreover,there exists $A>1$ depending on $\psi$ (hence, $A$ depends only on 
 $\Omega$) such that
 $$\dfrac{\lambda}{A}|\xi|^2\leq \dfrac{\partial G}{\partial r_{ij}}\xi_i\xi_j
 \leq A\Lambda |\xi|^2$$
 for all $\xi\in\R^N$.\\
 Now we only need to show
 $$\|D^2v(y,t)-D^2v(0,t_0)\|\leq C_{\epsilon,T}(|y|+|t-t_0|^{1 /2})^{\beta}$$
 for any $y\in\Gamma_1, t,t_0\in (\epsilon,T)$.\\
 By the implicit function theorem, we have\\
 $$v_{NN}=H(t,y,v,\dot{v},Dv, (v_{ij})_{j<N}).$$
 By the chain rule, we have
 $$|DH|\leq \dfrac{A}{\lambda}(\sup |DG|+1).$$
 Hence, there exists $B>0$ such that
 \begin{flushleft}
  $\begin{array}{ll}
 |v_{NN}(y,t)-v_{NN}(0,t_0)|&\leq 
 B(\sup\limits_{j<N} |v_{ij}(y,t)-v_{ij}(0,t_0)|+|\dot{v}(y,t)-\dot{v}(0,t_0)|\\
 &+|Dv(y,t)-Dv(0,t_0)|+|y|+|t-t_0|).
 \end{array}$
 \end{flushleft}
 Note that $\dot{v}|_{\Gamma_4\times (\epsilon,T)}=v_j|_{\Gamma_4\times (\epsilon,T)}=
 v_{ij}|_{\Gamma_4\times (\epsilon,T)}=0$ for $j<N$. Then we only need to show
 \begin{equation}\label{lembb.vN}
 |v_N(y,t)-v_N(0,t_0)|\leq C_{\epsilon,T}(|y|+|t-t_0|^{1/2})^{\beta},\\
 \end{equation}
 \begin{equation}\label{lembb.vNk}
 |v_{Nk}(y,t)-v_{Nk}(0,t_0)|\leq C_{\epsilon,T}(|y|+|t-t_0|^{1/2})^{\beta},
 \end{equation}
 for any $y\in\Gamma_1, t,t_0\in (\epsilon, T)$ and $k<N$.\\
 By \eqref{lembb.G}, we have\\
 \begin{equation}\label{lembb.end1}
 \dot{v}=\Delta v+f_1(t,y),
 \end{equation}
 where $\Delta$ is the Laplacian operator and $f_1(t,y)=G(t,y,v,Dv,D^2v)-\Delta v$. By
 the hypothesis of theorem,
 $\|f_1\|_{L^{\infty}}$ is bounded by a universal constant.\\
 Now we take the derivative of equation \eqref{lembb.G} in the direction $y_k$ and get that
 \begin{equation}\label{lembb.end2}
 \dot{v}_k=\sum\limits_{i,j=1}^N(v_k)_{ij}
 \dfrac{\partial G}{\partial r_{ij}}(t,y,v,Dv,D^2v)+f_2(t,y),
 \end{equation}
 where  
 $$
 f_2(t,y)=\frac{\partial G}{\partial y_k}(t,y,v,Dv,D^2v)
 +v_k\frac{\partial G}{\partial p}(t,y,v,Dv,D^2v)
 +\sum\limits_{l=1}^N v_{lk}\frac{\partial G}{\partial q_l}(t,y,v,Dv,D^2v).
 $$
 Then $\|f_2\|_{L^{\infty}}$ is bounded by a universal constant.\\
 Then  \cite[Lemma 7.32]{Lieb96} states that
 \begin{Lem}
 If $u\in C^{2;1}(B_4^+\times (0,T))$ satisfies
 $$
 |\dot{u}-\sum a_{ij}u_{ij}|\leq A_1,
 $$
 $$|u|\leq A_2x_N,
 $$
where $a_{ij}\in C(B_4^+\times (0,T))$ is such that
 $$\sup |a_{ij}|\leq B \mbox{ and}$$
  $$\lambda|\xi|^2\leq \sum a_{ij}\xi_i\xi_j\leq \Lambda |\xi|^2,$$
  then there are positive constants $\beta$ and $C$ determined only by $A_1, A_2, B, \lambda,
  \Lambda,\epsilon, T, N$ such that
  $$(\sup\limits_{U(y,t,R)}\frac{u}{x_N}-\inf\limits_{U(y,R)}\frac{u}{x_N})\leq
  CR^{\beta}\left(\sup\limits_{B^+_4\times (0,T)}\frac{u}{x_N}-
  \inf\limits_{B^+_4\times (0,T)}\frac{u}{x_N}+1\right),
  $$
 where $y\in B^+_1$, $2\epsilon<t<T-2\epsilon$, $R<\epsilon$ and 
 $U(y,t,R)=B^+_R(y)\times (t-R^2,t+R^2)$.
 \end{Lem}
 Applying this lemma to the equations \eqref{lembb.end1}
 and  \eqref{lembb.end2}, we obtain \eqref{lembb.vN} and \eqref{lembb.vNk}.
 \end{proof}
 \begin{Cor}
 There exists $C_{\epsilon,T}>0$ depending on 
$\lambda,\Lambda,\Omega,C,\epsilon,T$ and the upper bound of 
$\|\varphi\|_{C^4}+\|F\|_{C^1}+\|f\|_{C^1}$ such that\\
 $$\dfrac{|\dot{u}(x,t)-\dot{u}(x_0,t_0)|}{(|x-x_0|+|t-t_0|^{1/2})^{\beta}}
 \leq C_{\epsilon,T},
 \;\;\;\forall x,x_0\in\partial\Omega;
 \forall t,t_0\in (\epsilon,T).$$
 where $0<\beta<1$ is the constant in Lemma \ref{b-b}.
 \end{Cor}
 \begin{Lem}\label{b-i-dot}
  There exists $C_{\epsilon,T}>0$ depending on 
$\lambda,\Lambda,\Omega,C,\epsilon,T$ and the upper bound of 
$\|\varphi\|_{C^4}+\|F\|_{C^1}+\|f\|_{C^1}$ such that\\
 $$\dfrac{|\dot{u}(x,t)-\dot{u}(x_0,t_0)|}{(|x-x_0|+|t-t_0|^{1/2})^{\beta /2}}
 \leq C_{\epsilon,T},
 \;\;\;\forall x\in\Omega , x_0\in\partial\Omega;
 \forall t,t_0\in (\epsilon,T).$$
 where $0<\beta<1$ is the constant in Lemma \ref{b-b}.
 \end{Lem}
 \begin{proof}
 By equation \eqref{parabollic}, we have
 \begin{equation}\label{eq.Ca.b-i-dot.A}
 |\ddot{u}-\sum\limits\dfrac{\partial F}{\partial r_{ij}}\dot{u}_{ij}|=
 |f_t(t,x,u)+\dot{u}f_u(t,x,u)|\leq A,
 \end{equation}
 where $A>0$ is a universal constant.\\ 
 Let $x_0\in\partial\Omega$ and $t_0\in (2\epsilon,T)$. We can choose  coordinates
  $(x_j)_{1\leq j\leq N}$ so that $x_0=0$ and the positive $x_N$ axis is the interior normal 
  direction of $\partial\Omega$ at $x_0$. We also assume that near $x_0$, $\partial\Omega$ is 
  represented as a graph
 \begin{center}
 $x_N=P(x')=\sum\limits_{j,k<N} P_{jk}x_jx_k+O(|x'|^3),$
 \end{center}
   where $x'=(x_1,...,x_{N-1})$.\\
 Let $Q(x')=P(x')-|x'|^2$.  We consider
 $$
 v=K_1(x_N-Q(x'))^{\beta /2}+K_2((x_N-Q(x'))^2+(t_0-t))^{\beta /4}.
 $$
 We have
 \begin{flushleft}
 $\begin{array}{ll}
 \dfrac{\partial^2(x_N-Q(x'))^{\beta /2}}{\partial x_i \partial x_j}
 &=\dfrac{\beta(\beta-2)}{4}(x_N-Q(x'))^{\beta/2-2}
 \dfrac{\partial (x_N-Q(x'))}{\partial x_i}
 \dfrac{\partial (x_N-Q(x'))}{\partial x_j}
 \\
 &+\dfrac{\beta}{2}(x_N-Q(x'))^{\beta/2-1}
 \dfrac{\partial^2 (x_N-Q(x'))}{\partial x_i\partial x_j},
 \end{array}$
 \end{flushleft}
 and
 \begin{flushleft}
 $\dfrac{\partial^2 ((x_N-Q(x'))^2+t_0-t)^{\beta /4}}{\partial x_i \partial x_j}$\\
  $=\dfrac{\beta(\beta-4)}{4}((x_N-Q(x'))^2+t_0-t)^{\beta/4-2}
   (x_N-Q(x'))^2\dfrac{\partial (x_N-Q(x'))}{\partial x_i}
 \dfrac{\partial (x_N-Q(x'))}{\partial x_j} $\\
 $+\dfrac{\beta}{4}((x_N-Q(x'))^2+t_0-t)^{\beta/4-1}
 \dfrac{\partial^2 (x_N-Q(x'))^2}{\partial x_i \partial x_j}.$
 \end{flushleft}
 Hence, there exists $R>0$ satisfying, by $F_{r_{11}}\geq \lambda$,
 \begin{equation}\label{eq.Ca.b-i-dot.K1}
 \sum\limits_{i,j=1}^N \dfrac{\partial F}{\partial r_{ij}}
 \dfrac{\partial^2(x_N-Q(x'))^{\beta /2}}{\partial x_i \partial x_j}
 \leq \dfrac{\lambda\beta (\beta-2)}{6}(x_N-Q(x'))^{\beta /2-2}<0,
 \end{equation}
 and
 \begin{equation}\label{eq.Ca.b-i-dot.K2}
 \sum\limits_{i,j=1}^N\dfrac{\partial F}{\partial r_{ij}}
 \dfrac{\partial^2 ((x_N-Q(x'))^2+t_0-t)^{\beta /4}}{\partial x_ix_j}
= O(x_N-Q(x'))^{\beta /2-2}.
 \end{equation}
On the other hand, 
\begin{equation}\label{eq.Ca.b-i-dot.u.}
|\dot{u}-\dot{u}(0,t_0)|\;|_{\partial_P((\Omega\cap B_R)\times (\epsilon,t_0))}
=O(((x_N-Q(x'))^2+t_0-t)^{\beta /4}).
\end{equation}
 By \eqref{eq.Ca.b-i-dot.A},
 \eqref{eq.Ca.b-i-dot.K1}, \eqref{eq.Ca.b-i-dot.K2}, \eqref{eq.Ca.b-i-dot.u.},
  there exists $K_1,K_2>0$ such that
 \begin{center}
 $v|_{\partial_P((\Omega\cap B_R)\times
 (\epsilon,t_0))}\geq 
 \pm(\dot{u}-\dot{u}(0,t_0))|_{\partial_P((\Omega\cap B_R)\times (\epsilon,t_0))},$\\[6pt]
 $(\pm\ddot{u}-\dot{v})-\sum\dfrac{\partial F}{\partial r_{ij}}(\pm\dot{u}_{ij}-v_{ij})
 \leq A+\dfrac{K_1\lambda\beta(\beta-2)}{8}\leq 0.$
 \end{center}
 The comparison principle of parabolic type (\cite{Fried83}) states that
 \begin{Lem}
 Let $\Omega$ be a bounded domain of $\R^N$ and $T>0$. 
 Let $u,v\in C^{2;1}(\Omega\times (0,T])\cap C(\bar{\Omega}\times [0,T])$. Assume that
$$
\dfrac{\partial (u-v)}{\partial t}-\sum a_{ij}
\dfrac{\partial^2 (u-v)}{\partial x_i\partial x_j}-b.(u-v)\leq 0,
$$
where $a_{ij}, b\in C(\Omega\times (0,T))$, $(a_{ij}(x,t))$
are positive definite symmetric matrices and $b(z,t)<0$.
Then $(u-v) \leq max(0,\sup\limits_{\partial_P(\Omega\times (0,T))}(u-v)).$
 \end{Lem}
 Applying the comparison principle, we have
 $$
 (\dot{u}-\dot{u}(0,t_0))|_{(\Omega\cap B_R)\times (\epsilon,t_0)}\leq v|_{(\Omega\cap 
 B_R)\times (\epsilon,t_0)}.
 $$
 Hence there exists $K>0$ such that
 $$|\dot{u}(x,t)-\dot{u}(0,t_0)|\leq K(|x|+|t-t_0|^{1/2})^{\beta /2},$$
 where $x\in\Omega\times B_R$ and $\epsilon<t\leq t_0$.\\
Note that  $R$ is independent of $x_0$ and $K$ is independent of $t_0$. Then there exists
$C_{\epsilon,T}$ such that
$$\dfrac{|\dot{u}(x,t)-\dot{u}(x_0,t_0)|}{(|x-x_0|+|t-t_0|^{1/2})^{\beta /2}}
 \leq C_{\epsilon},
 \;\;\;\forall x\in\Omega , x_0\in\partial\Omega;
 \forall t,t_0\in (2\epsilon,T).$$
 \end{proof} 
 \begin{Lem}\label{b-i 1}
 There exists $C_{\epsilon,T}>0$ depending on 
$\lambda,\Lambda,\Omega,C,\epsilon,T$ and upper bound of 
$\|\varphi\|_{C^4}+\|F\|_{C^1}+\|f\|_{C^2}$ such that\\
 $$u_{\xi\xi}(x,t)-u_{\xi\xi}(x_0,t_0)
 \leq C_{\epsilon,T}(|x-x_0|+|t-t_0|^{1/2})^{\beta /2}$$
for any  $\xi\in\R^N,|\xi|=1,x\in\Omega, x_0\in\partial\Omega, \epsilon<t,t_0<T$.
Where $0<\beta<1$ is the constant in Lemma \ref{b-b}.
 \end{Lem}
 \begin{proof}
 By the equation \eqref{parabollic}, we have
 $$\dot{u}_{\xi\xi}-\sum\dfrac{\partial F }{\partial r_{ij}}(u_{\xi\xi})_{ij}
 -f_u.u_{\xi\xi}
 =\sum\dfrac{\partial^2 F}{\partial r_{ij}\partial r_{kl}}(u_{\xi})_{ij}(u_{\xi})_{kl}
 +O(1)\leq O(1)$$
 By Lemma \ref{b-b}, we also obtain 
 $$ (u_{\xi\xi}(x,t)-u_{\xi\xi}(x_0,t_0))|_{\partial_P(\Omega\times (\epsilon,T))}
 =O(|x-x_0|+|t-t_0|^{1/2})^{\beta /2})$$
Then, the proof of Lemma \ref{b-i 1} is similar to the proof 
of Lemma \ref{b-i-dot} with the same type of fuction $v$.
 \end{proof}
 
 \begin{Lem}\label{b-i 2}
  There exists $C_{\epsilon,T}>0$ depending on 
$\lambda,\Lambda,\Omega,C,\epsilon,T$ and upper bound of 
$\|\varphi\|_{C^4}+\|F\|_{C^1}+\|f\|_{C^2}$ such that\\
  $$\|D^2u(x,t)-D^2u(x_0,t_0)\|\leq C_{\epsilon,T}(|x-x_0|+|t-t_0|^{1/2})^{\beta /2}$$
  for any $x\in\Omega, x_0\in\partial\Omega, \epsilon<t,t_0<T$,
where $0<\beta<1$ is the constant in Lemma \ref{b-b}.
 \end{Lem}
 \begin{proof}
 Let $\lambda_1,...,\lambda_N$ be eigenvalues of $D^2u(x,t)-D^2u(x_0,t_0)$. We have\\
 $$\|D^2u(x,t)-D^2u(x_0,t_0)\|\leq\sum|\lambda_i|.$$
 Moreover,
 \begin{flushleft}
 $\begin{array}{ll}
 \dot{u}(x,t)-f(t,x,u(x,t))&=F(D^2u(x,t))\\ &\leq F(D^2u(x_0,t_0))
 +\Lambda\sum\limits_{\lambda_i>0}\lambda_i+\lambda \sum\limits_{\lambda_i<0}\lambda_i\\
 &=\dot{u}(x_0,t_0)-f(t_0,x_0,u(x_0,t_0))
 +\Lambda\sum\limits_{\lambda_i>0}\lambda_i+\lambda \sum\limits_{\lambda_i<0}\lambda_i.\\
 \end{array}$
 \end{flushleft}
 Hence, by Lemma \ref{b-i-dot} , we have\\
 $$\Lambda\sum\limits_{\lambda_i>0}|\lambda_i|\geq 
 \lambda \sum\limits_{\lambda_i<0}|\lambda_i|
 -A(|x-x_0|+|t-t_0|^{1/2})^{\beta /2},$$
 where  $A>0$ is a universal constant.\\
 Then
 $$\|D^2u(x,t)-D^2u(x_0,t_0)\|\leq
  \frac{\Lambda+\lambda}{\lambda}\sum\limits_{\lambda_i>0}|\lambda_i|
 +\frac{A}{\lambda}(|x-x_0|+|t-t_0|^{1/2})^{\beta /2}.$$
 Note that
 $$
 \sum\limits_{\lambda_i>0}|\lambda_i|\leq N\max\{0,\lambda_1,...\lambda_N\}
 \leq N \max\{\sup\limits_{|\xi|=1}(u_{\xi\xi}(x,t)-u_{\xi\xi}(x_0,t_0)), 0\}.$$
 By Lemma \ref{b-i 1}, there exists  $C_{\epsilon,T}>0$ depending on 
 $\lambda,\Lambda, \Omega,C,\epsilon, T$ and 
 upper bound of $\|\varphi\|_{C^4}+\|F\|_{C^1}+\|f\|_{C^2}$ such that \\
  $$\|D^2u(x,t)-D^2u(x_0,t_0)\|\leq C_{\epsilon,T}(|x-x_0|+|t-t_0|^{1/2})^{\beta /2}$$
  for any $x\in\Omega, x_0\in\partial\Omega, \epsilon<t,t_0<T$.\\
 \end{proof}
 
 \begin{proof}[Proof of Theorem \ref{main.C2alpha}]
 We need to show that
 \begin{equation}\label{C2alpha.gol}
 \|D^2u(x,t_1)-D^2u(y,t_2)\|\leq C (|x-y|+|t_1-t_2|^{1 /2})^{\gamma},
 \end{equation}
 where $x,y\in\Omega$, $2\epsilon<t_1,t_2<T-\epsilon$. $C$ and
  $\gamma$ are universal constants.\\
 We can assume that $d_x:=d(x,\partial\Omega)\geq d_y:=d(y,\partial\Omega)$.\\
 If $|x-y|^2 +|t_1-t_2|\leq \min\{\frac{d_x^2}{4},\frac{\epsilon}{2}\})$, we denote
 $$v(\xi,t)=\dfrac{1}{a^2}\left( u(x+a.\xi, t_1+a^2t)
 -u(x,t_1)-a\sum u_k(x,t_1)\xi_k\right),$$
 where $a=\min\{ d_x, \epsilon^{1/2} \}$. Then $v\in C^{\infty}
 (\mathbb{B}\times (-1,1))$ satisfies
 $$\dot{v}=F(D^2v)+f(t_1+a^2t,x_1+a\xi,u(x_1+a\xi,t_1+a^2t))=F(D^2v)+\tilde{f}(t,\xi).$$
  It follows from the interior estimate
 (see the theorem 14.7 and the lemma 14.8 of \cite{Lieb96}) that 
 \begin{center}
 $\|v\|_{C^{2,\gamma}(\mathbb{B}_{1/2}\times (-1/2, 1/2))}
 \leq A (\|v\|_{C^2(\mathbb{B}\times (-1, 1))}+1),$
 \end{center}
 where $A$ is universal, $\gamma=\min\{\alpha,\beta /2 \}$, $\beta$ is the constant in Lemma
 \ref{b-b} and $\alpha$ is the constant in Theorem 14.7 of \cite{Lieb96}.\\
 Moreover
 \begin{flushleft}
 $\begin{array}{ll}
 |v(\xi,t)|&\leq \dfrac{|u(x+a\xi,t_1+a^2t)-u(x+a\xi,t_1)|}{a^2}\\[6pt]
 &+\dfrac{|u(x+a\xi,t_1)-u(x,t_1)-a\sum u_k(x,t_1)\xi_k|}{a^2}\\[6pt]
 &\leq \sup |\dot{u}|+\sup \|D^2 u\|,\\[6pt]
  \end{array}$
 $|\dot{v}(\xi,t)|=|\dot{u}(x+a\xi,t_1+a^2t)|\leq \sup |\dot{u}|,$\\[6pt]
  $\|D^2v(\xi,t)\|=\|D^2u(x+a\xi,t_1+a^2t)\|\leq \sup \|D^2u\|.$\\
 \end{flushleft}
 Hence
 \begin{center}
 $\|v\|_{C^{2,\gamma}(\mathbb{B}_{1/2}\times (-1/2, 1/2))}
 \leq B,$
 \end{center}
 where $B$ is universal.\\
 Then
 $$\|D^2u(x,t_1)-D^2u(y,t_2)\|\leq B (|x-y|+|t_1-t_2|^{1 /2})^{\gamma}.
 $$
 If $|x-y|^2+|t_1-t_2|\geq \frac{\epsilon}{2}$, then
 $$
 \|D^2u(x,t_1)-D^2u(y,t_2)\|\leq 2(\frac{\epsilon}{2})^{-\gamma/2}
 (\sup\|D^2u\|)(|x-y|+|t_1-t_2|^{1 /2})^{\gamma}.
  $$
 If $\frac{\epsilon}{2}>|x-y|^2+|t_1-t_2|\geq \frac{d_x^2}{4}$,
it follows from Lemma \ref{b-i 2} that
 \begin{flushleft}
 $\begin{array}{ll}
 \|D^2u(x,t_1)-D^2u(y,t_2)\|&\leq \|D^2u(x,t_1)-D^2u(x_0,t_1)\|
 +\|D^2u(x_0,t_1)-D^2u(y,t_2)\|\\
 &\leq C_{\epsilon,T}(|x-x_0|^{\beta/2}+(|x_0-y|+|t_1-t_2|^{1/2})^{\beta/2})\\
 &\leq C (|x-y|+|t_1-t_2|^{1/2})^{\beta /2}\\
 &\leq C (|x-y|+|t_1-t_2|^{1/2})^{\gamma}\\
 \end{array},$
 \end{flushleft}
 where $C_{\epsilon,T}$ is the constant in Lemma \ref{b-i 2}, $x_0\in\partial\Omega$
 satisfies $d_x=|x-x_0|$ and $C$ is universal.
 \end{proof}
 \subsection{Higher regularity}
 \begin{flushleft}
 Let $g\in C^{k+1,\alpha}(\bar{\Omega}\times [0,T))$, where $k\geq 0, 0<\alpha<1$.
 Let $F$ be a function defined on $Mat(N\times N,\R)\times\bar{\Omega}\times  [0,T)$ 
 such that $F(.,x,t)$ is concave and satisfies \eqref{F}.  Assume that 
 $F\in C^{k+2;k+1,\alpha}(Mat(N\times N,\R)\times\bar{\Omega}\times  [0,T))$, i.e., 
 the derivaties $D_r^iD_x^jD_t^lF$ are continuous for all $|i|\leq k+2, |j|+2l\leq k+1$ 
 and satisfy
 \end{flushleft}
  $$\|F\|_{C^{k+2;k+1,\alpha}(Mat(N\times N,\R)\times\bar{\Omega}\times  [0,T))}=
  \sum\limits_{|i|\leq k+2}\sup\limits_{r\in Mat(N\times N,\R)} 
  |D_r^iF(r,.)|_{C^{k+1,\alpha}(\bar{\Omega}\times  [0,T))}<\infty.$$
  We  consider the $C^{k+3,\alpha}$ regularity of a solution $u$ of the equation
 \begin{equation}
 \dot{u}=F(D^2u,x,t)+g(x,t).
 \end{equation}
 The following boundary estimates hold:
  \begin{Prop}\label{Ckalpha.prop1}
 Let $x_0\in\partial\Omega$, $k\geq 0$, $r>0$ and 
 $u\in C^{\infty}((\Omega\cap B_r(x_0))\times (0,T))\cap
 C^{k+2,\alpha}((\Omega\cap B_r(x_0))\times (0,T))$
 be a solution of 
 \begin{equation}\label{Evans_Krylov}
 \begin{cases}
 \dot{u}=F(D^2u,x,t)+g(x,t)\,\mbox{ on }\, (\Omega\cap B_r(x_0))\times (0,T),\\
 u=\varphi\,\mbox{ on }\, (\partial\Omega\cap B_r(x_0))\times (0,T),
 \end{cases}
 \end{equation}
  where $\varphi\in C^{k+3,\alpha}(\bar{\Omega}\times (0,T)$. Then
   there exists $r'\in (0,r)$
 depending on $r, \Omega$  such that 
  $u\in C^{3+k,\alpha}((\Omega\cap B_{r'}(x_0))\times (\epsilon,T'))$ for any
   $0<\epsilon<T'<T$. Moreover
  $$\|u\|_{C^{k+3,\alpha}((\Omega\cap B_{r'}(x_0))\times (\epsilon,T'))}
  \leq K,$$
  where $K>0$ depends on $\lambda, \Lambda, \alpha, \Omega,\epsilon, T',T, r, r', 
  \|u\|_{C^{k+2,\alpha}}, \|F\|_{C^{k+2;k+1,\alpha}}, \|g\|_{C^{k+1,\alpha}}$,
  \newline
  $\|\varphi\|_{C^{k+3,\alpha}}$.
 \end{Prop}
 This regularity  is proved, for example, in \cite{Lieb96} (or \cite{GT83} ,
 \cite{CC95} for the elliptic version). For the reader's convenience, 
 we recall the arguments here.
 \begin{proof}
 Using a smooth diffeomorphism (as proof of Lemma \ref{b-b}), we can replace 
 $\Omega\cap B_r(x_0)$ by $B_4^+$ and replace $\partial\Omega\cap B_r(x_0)$ by
 $\Gamma_4$. We need to show that
  $u\in C^{k+3,\alpha}(B_1^+\times (\epsilon,T')).$\\
  Let $h>0$ be small and $e_l$ be the $l^{th}$ vector of the standard basis of $R^N$,
   $l<N$. We define
  \begin{flushleft}
  $\begin{array}{ll}
  a_{ij}^h(x,t) &=\int\limits_0^1\dfrac{\partial F}{\partial r_{ij}}
  (sD^2u(x+he_l,t)+(1-s)D^2u(x,t), x+she_l,t)ds ,\\
  g^h(x,t)&=\dfrac{g(x+he_l,t)-g(x,t)}{h} ,\\
  G^h(x,t)&=\int\limits_0^1 F_l(sD^2u(x+he_l,t)+(1-s)D^2u(x,t),x+she_l,t)ds,\\
  \varphi^h(x,t)&=\dfrac{\varphi(x+he_l,t)-\varphi(x,t)}{h},\\
  v^h(x,t)&=\dfrac{u(x+he_l,t)-u(x,t)}{h}.
  \end{array}$
\end{flushleft} 
 
For the convenience, we denote $Q_a=B_a^+\times (0,T)$ for any $a>0$. 
Then $$\|a_{ij}^h\|_{C^{k,\alpha}(Q_2)}+ \|g^h\|_{C^{k,\alpha}(Q_2)}+ \|G^h\|_{C^{k,\alpha}
(Q_2)}+  \|v^h\|_{C^{k+1,\alpha}(Q_2)}+ 
 \|\varphi^h\|_{C^{k+2,\alpha}(Q_2)}<A,$$ where $A>0$ depends
 only on  $\|u\|_{C^{k+2,\alpha}(Q_4)}, \|F\|_{C^{k+2;k+1,\alpha}(Q_4)},
  \|g\|_{C^{k+1,\alpha}(Q_4)},
  \|\varphi\|_{C^{k+3,\alpha}(Q_4)}.$ Moreover, 
  \begin{equation}
  \begin{cases}
  \dot{v}^h=\sum a_{ij}^hv_{ij}^h+g^h+G^h \,\mbox{ on } Q_2,\\
  v^h=\varphi^h \,\mbox{ on } \Gamma_2\times (0,T).
  \end{cases}
  \end{equation}
  If $k=0$, using a cutoff function and applying Schauder's global estimates 
 ( \cite{Fried83},page 65), we have
  \begin{equation}\label{quynap.smooth}
  \|v^h\|_{C^{k+2,\alpha}(B_1^+\times (\epsilon, T'))}\leq C,
  \end{equation}
  where $C>0$ depends on $A$ and $\epsilon, T'$.
  
  If $k>0$ and Proposition \ref{Ckalpha.prop1} is verified for $k-1$, 
  then applying the case $k-1$,
  we also obtain \eqref{quynap.smooth}.
  
  It follows that $u_l\in C^{k+2,\alpha}(B_1^+\times (\epsilon, T'))$ with 
  $\|u_l\|_{C^{k+2,\alpha}(B_1^+\times (\epsilon, T'))}\leq C.$
  
  By the same method, we can also show that  
  $\|\dot{u}\|_{C^{k+1,\alpha}(B_1^+\times (\epsilon, T'))}\leq C$.
  It remains to prove $\|u_{NNN}\|_{C^{k,\alpha}(B_1^+\times (\epsilon, T'))}\leq C$.
  On $B_1^+\times (\epsilon, T')$, we have
  \begin{center}
  $
  \dot{u}_N=\sum (\dfrac{\partial F}{\partial r_{ij}}(D^2u,x,t))u_{ijN}
  +F_N(D^2u,x,t)+g_N(x,t).
  $
  \end{center}
  Then
  \begin{center}
  $u_{NNN}=\dfrac{1}{\partial F /\partial r_{NN}}
  \left(\dot{u}_N-\sum\limits_{(i,j)\neq (N,N)}\dfrac{\partial F}{\partial r_{ij}}u_{ijN}
  -g_N \right).$
  \end{center}
  Note that $\frac{\partial F}{\partial r_{NN}}\geq \lambda >0$. 
  Hence, $u_{NNN}\in C^{k,\alpha}(B_1^+\times (\epsilon, T'))$ and 
  $\|u_{NNN}\|_{C^{k,\alpha}(B_1^+\times (\epsilon, T'))}$ 
  is bounded by a universal constant.
 \end{proof}
 Using the method of the proof above, we also obtain the interior estimates
 \begin{Prop}\label{Ckalpha.prop2}
 Let $x_0\in\Omega$ and $0<r<d(x_0, \partial\Omega)$. Let 
 $u\in C^{k+2,\alpha}(B_r(x_0)\times(0,T))$ be a solution of 
 \begin{equation}
 \dot{u}=F(D^2u,x,t)+g(x,t) \,\mbox{ on }\, B_r(x_0).
 \end{equation}
 Then $u\in C^{k+3,\alpha}(B_{r/2}(x_0)\times (\epsilon,T'))$ for any 
 $0<\epsilon<T'<T$. Moreover
 $$\|u\|_{C^{k+3,\alpha}(B_{r/2}(x_0)\times (\epsilon,T'))}\leq C,$$
 where $C>0$ depends on $\lambda, \Lambda, \alpha,\epsilon, T',T, r, 
  \|u\|_{C^{k+2,\alpha}}, \|F\|_{C^{k+2;k+1,\alpha}}, \|g\|_{C^{k+1,\alpha}}$.
 \end{Prop}
 Combining Proposition \ref{Ckalpha.prop1} and Proposition \ref{Ckalpha.prop2}, 
 we have the following
 \begin{Prop}\label{Cinftyalpha}
 Let $F,f,\varphi$ be functions defined as \ref{C2a.subsec}. Assume that
 $u\in C^{2,\alpha}(\Omega\times (0,T))$ is a solution of 
 \begin{equation}\label{Cinfty }
 \begin{cases}
 \dot{u}=F(D^2u)+f(t,x,u) \,\mbox{ on }\, \Omega\times (0,T),\\
 u=\varphi\,\mbox{ on }\, \partial\Omega\times (0,T).
 \end{cases}
 \end{equation}
 Then $u\in C^{\infty}(\bar{\Omega}\times (0,T))$.
 \end{Prop}
 \section{Proof of the main theorem}
 \label{pfmain}
 We recall the main theorem:
 \begin{The}[Main theorem]\label{maintheo}
 Let $\Omega$ be a bounded smooth strictly pseudoconvex domain of $\C^n$ and $T\in (0,\infty ]$.
 Let $u_0$ be a bounded plurisubharmonic function defined on a neighbourhood $\tilde{\Omega}$ of
 $\overline \Omega$. Assume that $\varphi\in C^{\infty}(\bar{\Omega}\times [0,T))$ and $f
 \in C^{\infty}([0,T)\times \bar{\Omega}\times \R )$  satisfying
 \begin{itemize}
 \item[(i)] $f_u\leq 0.$
 \item[(ii)] $\varphi(z,0)=u_0(z)$ for $z\in\partial\Omega$.
 \end{itemize}
 Then there exists a unique function $u\in C^{\infty}(\bar{\Omega}\times (0,T))$ such that
 \begin{equation}\label{psh.cmdlc}
 u(.,t) \mbox{ is a strictly plurisubharmonic function on } \Omega, \;\;\forall t\in (0,T),
 \end{equation}
 \begin{equation}\label{KRF.cmdlc}
\dot{u}= \log\det (u_{\alpha\bar{\beta}})+f(t,z,u) \mbox{ on } \Omega\times (0,T),
\end{equation}
\begin{equation}\label{bien.cmdlc}
u=\varphi \mbox{ on } \partial\Omega\times (0,T),
\end{equation}
\begin{equation}\label{initial.cmdlc}
\lim\limits_{t\rightarrow 0} u(z,t)= u_0(z)\;\;\forall z\in\bar{\Omega}.
 \end{equation}
 Moreover,  $u\in L^{\infty}(\bar{\Omega}\times [0,T'))$ for any $0<T'<T$, and
 $u(.,t)$  also converges to $u_0$ in capacity when $t\rightarrow 0$.\\
 If $u_0\in C(\tilde{\Omega})$ then $u\in C(\bar{\Omega}\times [0,T))$.
 \end{The}
 \begin{proof}
 \begin{flushleft}
Replacing $T$ by $0<T'<T$, we can assume that $T<\infty$ and there exists 
$C_{\varphi}$ such that
\end{flushleft}
\begin{equation}\label{cphi.cmdlc}
\|\varphi\|_{C^4(\Omega\times (0,T))}\leq C_{\varphi}.
\end{equation}
We can also assume that $\|f\|_{C^2([0,T)\times\bar{\Omega}\times [-M,M])}<\infty$ for any
$M>0$.\\
 \textbf{Existence of a solution.}\\
 Using the convolution of $u_0+\frac{|z|^2}{m}$ with smooth kernels, 
 we can take $u_{0,m}\in C^{\infty}
 (\bar{\Omega})$ such that
 $$u_{0,m}\searrow u_0,$$
 $$dd^cu_{0,m}\geq \frac{1}{m}dd^c|z|^2.$$
 Note that $u_0|_{\partial\Omega}$ is continuous. Then 
 \begin{equation}\label{eq.deltam}
 \delta_m=\sup\limits_{z\in\partial\Omega}(u_{0,m}(z)-u_0(z))
 \stackrel{m\rightarrow\infty}{\longrightarrow} 0.
 \end{equation}
 We define $g_m\in C^{\infty}(\bar{\Omega})$ and
  $\varphi_m\in C^{\infty}(\bar{\Omega}\times [0,T))$ by\\
  $$g_m=-\log\det (u_{0,m})_{\alpha\bar{\beta}}+f(0,z,u_{0,m}),$$
  $$\varphi_m=\zeta(\frac{t}{\epsilon_m}) (tg_m+u_{0,m})+(1-\zeta(\frac{t}
  {\epsilon_m}))\varphi,$$
  where $\zeta$ is a smooth funtion on $\R$ such that $\zeta$ is decreasing,
   $\zeta|_{(-\infty,1]}=1$ and $\zeta|_{[2,\infty)}=0$. $\epsilon_m>0$ are 
 chosen such that  the sequences $\{\epsilon_m\}$,
 $\{\epsilon_m\sup|g_m|\}$ are decreasing to $0$
 and $\zeta(\frac{t}{\epsilon_m})(u_{0,m}(z)-\varphi (z,t))\geq 0$ for any $m$.\\
  Then $\varphi_m$ converges pointwise to $\varphi$ on $\partial\Omega\times [0,T)$ and for any 
  $0<\epsilon<T$, there exists $m_{\epsilon}>0$ such that 
  $\varphi_m|_{\bar{\Omega}\times (\epsilon,T)}=\varphi|_{\bar{\Omega}\times (\epsilon,T)},
 \; \forall m>m_{\epsilon}.$\\
   Moveover,
   \begin{center}
 $\varphi_m(z,0)=u_{0,m}(z)\;\; ,$\\
 $\dot{\varphi}_m=\log\det(u_{0,m})_{\alpha\bar{\beta}}+f(t,z,u_{0,m}),$\\
 \end{center}
 where $(z,t)\in\partial\Omega\times \{ 0 \}$.\\
  By the theorem of Hou-Li, there exists $u_m\in C^{\infty}(\Omega\times (0,T))
 \cap C^{2;1}(\bar{\Omega}\times [0,T))$ satisfying 
 \begin{equation}\label{KRFM}
\begin{cases}
\begin{array}{ll}
\dot{u}_m=\log\det (u_m)_{\alpha\bar{\beta}}+f(t,z,u_m)\;\;\;&\mbox{on}\;\Omega\times (0,T),\\
u_m=\varphi_m&\mbox{on}\;\partial\Omega\times [0,T),\\
u_m=u_{0,m}&\mbox{on}\;\bar{\Omega}\times\{ 0\}.\\
\end{array}
\end{cases}
\end{equation}
  Applying Corollary \ref{compa} for $u_1$ and $u_m$,  we see that the functions
   $u_m$ are uniformly bounded
  by a constant $C_u>0$.
 Then we can assume that $\|f\|_{C^2((0,T)\times\Omega\times \R)}\leq C_f$.
  Applying Theorem \ref{Main} on $\Omega\times (\frac{\epsilon}{2},T)$, we obtain
 $$\|u_m\|_{C^2(\Omega\times (\epsilon,T))}\leq C,$$
 where $C=C(\epsilon, T,\Omega,C_f,C_{\varphi}, C_u)$, $m$ is large enough.\\
 It follows from the $C^{2,\alpha}$ estimates in Section \ref{C2a} that for any 
  $0<\epsilon<T'<T$, there exist
 $M=M(\epsilon,T', C,\Omega,C_{\varphi},C_f)$ and $0<\gamma<1$  such that
 $$\|u_m\|_{C^{2,\gamma}(\bar{\Omega}\times (\epsilon, T))}\leq M.$$
 By Ascoli's theorem, there exists $u\in C^{2,\gamma/2}(\bar{\Omega}\times (0,T))$ 
 such that
 \begin{equation}\label{defineu.cmdlc}
 u_{m_k}\stackrel{C^{2,\gamma/2}(\bar{\Omega}\times (\epsilon,T))}{\longrightarrow}u.
 \end{equation}
 Thus $u$ satisfies \eqref{psh.cmdlc}, \eqref{KRF.cmdlc} and \eqref{bien.cmdlc}. By 
  Proposition \ref{Cinftyalpha} 
 we have $u\in C^{\infty}(\bar{\Omega}\times (0,T))$.\\
  Clearly, $u$ is bounded.  We need to show the convergence of $u(.,t)$ when $t\rightarrow 0$.\\
  \textit{Step 1: $\liminf\limits_{t\rightarrow 0} u(z,t)\geq u_0(z).$}\\
  By \eqref{defineu.cmdlc},
 there exists a subsequence of $(u_m)$, also denoted by $(u_m)$, which converges pointwise to $u$
 on $\bar{\Omega}\times (0,T)$.\\
 For any $a>0$, there exists $A>0$ such that $\forall m>0, v_m=u_{0,m}+a\rho-At$ satisfies
\begin{equation}\label{eq.sualoi1.cmdlc}
 \begin{cases}
 \dot{v}_m\leq \log\det (v_m)_{\alpha\bar{\beta}}+f(t,z,v_m),\\
 v_m|_{\partial_P(\Omega\times (0,T))}\leq u_m|_{\partial_P(\Omega\times (0,T))}+
 \epsilon_m\sup|g_m|+\delta_m,\\
 \end{cases}
 \end{equation}
 where $\rho\in C^{\infty}(\bar{\Omega})$ is a non-positive strictly plurisubharmonic function
 on $\Omega$. \\
 It follows from Corollary \ref{compa} that 
 $$v_m\leq u_m+\epsilon_m\sup|g_m|+\delta_m.$$
 Hence
 \begin{equation}\label{convergeq.cmdlc}
 u(z,t)\geq \lim\limits_{m\rightarrow \infty}(v_m(z,t)-\epsilon_m\sup|g_m|-\delta_m)
 =u_0(z)+a\rho (z)-At.
 \end{equation}
 Then we have
 $$
 \liminf\limits_{t\rightarrow 0} u(z,t)\geq u_0(z)+a\rho (z).
 $$
 When $a\rightarrow 0$, we obtain
 \begin{equation}\label{liminf.cmdlc}
 \liminf\limits_{t\rightarrow 0} u(z,t)\geq u_0(z).
 \end{equation}
\textit{Step 2: $\limsup\limits_{t\rightarrow 0} u(z,t)\leq u_0(z).$}\\
Let $\epsilon>0$. Assume that $m_0>0$ satisfies $\epsilon_{m_0}\sup|g_{m_0}|\leq\epsilon$.\\
For any $m>k>m_0$, we have
\begin{flushleft}
$\begin{array}{ll}
u_{0,m}-u_{0,k}&\leq 0;\\
\varphi_m-\varphi_k&=\zeta(\frac{t}{\epsilon_m})(u_{0,m}-\varphi)
-\zeta (\frac{t}{\epsilon_k})(u_{0,k}-\varphi)\\
&+tg_m\zeta(\frac{t}{\epsilon_m})-tg_k\zeta(\frac{t}{\epsilon_k})\\
&\leq \zeta(\frac{t}{\epsilon_k})(u_{0,m}-\varphi)-\zeta (\frac{t}{\epsilon_k})
(u_{0,k}-\varphi)+2\epsilon \\
&\leq \zeta(\frac{t}{\epsilon_k})(u_{0,m}-u_{0,k})+2\epsilon \\
&\leq 2\epsilon .
\end{array}$
\end{flushleft}
It follows Corollary \ref{compa} that 
$$u_m\leq u_k+2\epsilon .$$
Hence
\begin{equation}\label{converleq.cmdlc}
u(z,t)=\lim\limits_{m\rightarrow \infty} u_m(z,t)\leq u_k(z,t)+2\epsilon .
\end{equation}
Then we have
$$\limsup\limits_{t\rightarrow 0} u(z,t)\leq u_{0,k}(z)+2\epsilon .$$
When $k\rightarrow\infty$ and $\epsilon\rightarrow 0$, we obtain
\begin{equation}\label{limsup.cmdlc}
\limsup\limits_{t\rightarrow 0} u(z,t)\leq u_0(z).
\end{equation}
Combining \eqref{liminf.cmdlc} and \eqref{limsup.cmdlc}, we obtain \eqref{initial.cmdlc}.\\
\textit{Step 3: Convergence in capacity.}\\
 The bounded plurisubharmonic function $u_0$ is continuous outside sets of arbitrarily small capacity.
  Then the convergence in capacity is implied by \eqref{convergeq.cmdlc}, \eqref{converleq.cmdlc} and 
 Hartogs lemma (Lemma 90 of \cite{Ber13}) .\\ 
If $u_0\in C(\tilde{\Omega})$ then $u_{ 0,m}$ and $\varphi_m$  converge uniformly,
 respectively, to $u_0$ and $\varphi$. It follows Corollary \ref{compa} that $u_m$  converges 
 uniformly to  $u$. So $u$ is continuous on $\bar{\Omega}\times [0,T)$.\\
 \textbf{Uniqueness of the solution.}\\
 Let $u, v\in C^{\infty}(\bar{\Omega}\times (0,T))$ be functions satisfying \eqref{psh.cmdlc},
 \eqref{KRF.cmdlc}, \eqref{bien.cmdlc}, \eqref{initial.cmdlc}. Let $\epsilon>0$. We need to 
 show that $u\leq v+(t+3)\epsilon$.\\
 \textit{Step 1.} $\exists A>0 , v(z,t)\geq u_0(z)-\epsilon-At.$\\
 For $m>0$, we denote $v_m(z,t)=v(z,t+\frac{1}{m})$. Then $v_m$ is the solution of 
 \begin{equation}
 \begin{cases}
 \dot{v}_m=\log\det(v_m)_{\alpha\bar{\beta}}+f(t+\frac{1}{m},z,v_m)\mbox{ on } \Omega\times (0,T-
 \frac{1}{m}),\\
 v_m(z,t)=\varphi(z,t+\frac{1}{m}) \mbox{ on } \partial\Omega\times (0,T-\frac{1}{m}).
 \end{cases}
 \end{equation}
 Let $\rho\in C^{\infty}(\bar{\Omega})$ be a non-positive strictly plurisubharmonic function 
 on $\Omega$ such that $\inf\rho=-1$. Then there exists $A>0$ depending only on
 $\epsilon, \rho, \|\varphi\|_{C^1}, \sup f(t,z,\sup \varphi)$ such that
 \begin{equation}
 \begin{cases}
 \dot{w}_m\leq \log\det(w_m)_{\alpha\bar{\beta}}+f(t+\frac{1}{m},z,w_m)\mbox{ on } \Omega\times (0,T-
 \frac{1}{m}),\\
 w_m(z,t)\leq \varphi(z,t+\frac{1}{m}) \mbox{ on } \partial\Omega\times (0,T-\frac{1}{m}),
 \end{cases}
 \end{equation}
 where $w_m =v(z,\frac{1}{m})+\epsilon\rho-At$.\\
 Applying Corollary \ref{compa}, we have $v_m\geq w_m$. When $m\rightarrow\infty$, we obtain
 $$v(z,t)\geq u_0(z)+\epsilon\rho (z)-At\geq u_0(z)-\epsilon-At.$$
 \textit{Step 2. }$\exists m_0>0, \forall m>m_0, \exists k_m>m, v(z,\frac{1}{m})\geq -3\epsilon
 +u(z,\frac{1}{k_m})$.\\
 Step 1 implies that $v$ is bounded. Then we can assume that 
 $\|f\|_{C^2([0,T)\times\bar{\Omega}\times\R)}<\infty$.\\
 By step 1, we have
 $$v(z,\frac{1}{m})+\epsilon+\frac{A}{m}\geq u_0(z)=\lim\limits_{t\rightarrow 0}u(z,t).$$
 Applying Hartogs lemma, for any $K\Subset\Omega$ there exists $k_{m,K}>m$ such that
 \begin{equation}\label{trongK.step2.uni}
 u(z,\frac{1}{k_{m,K}})\leq v(z,\frac{1}{m})+2\epsilon+\frac{A}{m}\;\;\forall z\in K.
\end{equation}  
Let $m_0\geq \frac{1}{\epsilon}\max\{1, A, \|f\|_{C^2},\|h\|_{C^2}\}$, where
$h\in C^{\infty}(\bar{\Omega}\times [0,T))$ is a spatial harmonic function such that
$h|_{\partial\Omega\times (0,T)}=\varphi|_{\partial\Omega\times (0,T)}$.\\
For any $m>m_0$, let $K=K_m\Subset\Omega$ such that 
$$v(z,\frac{1}{m})+\epsilon\geq h(z,\frac{1}{m})\;\;\forall z\in\Omega\setminus K.$$
Let $k_m=k_{m,K_m}$. Then
\begin{equation}\label{ngoaiK.step2.uni}
v(z,\frac{1}{m})\geq -2\epsilon + h(z,\frac{1}{k_m})\geq -2\epsilon +u(z,\frac{1}{k_m})\;\;\forall
z\in \Omega\setminus K.
\end{equation}
Combining \eqref{trongK.step2.uni} and \eqref{ngoaiK.step2.uni}, we obtain
$$v(z,\frac{1}{m})\geq -3\epsilon+u(z,\frac{1}{k_m})\;\;\forall z\in\Omega.$$
\textit{Step 3. Conclusion.}
\begin{flushleft}
Let $u_m(z,t)=u(z,t+\frac{1}{k_m})-\epsilon t$. For $m>m_0$, we have
\end{flushleft}
\begin{equation}
\begin{cases}
\dot{v}_m=\log\det (v_m)_{\alpha\bar{\beta}}+f(t+\frac{1}{m},z,v_m)\geq 
\log\det (v_m)_{\alpha\bar{\beta}}+f(t+\frac{1}{k_m},z,v_m)-\epsilon,\\
\dot{u}_m\leq \log\det (u_m)_{\alpha\bar{\beta}}+f(t+\frac{1}{k_m},z,u_m)-\epsilon.
\end{cases}
\end{equation}
Applying Corollary \ref{compa}, we have
$$(u_m-v_m)\leq\sup\limits_{\partial_P(\Omega\times (0,T-\frac{1}{m}))}(u_m-v_m)
\leq 3\epsilon $$
When $m\rightarrow \infty$, we have
$$u(z,t)-v(z,t)-\epsilon t=\lim\limits_{m\rightarrow \infty}(u_m(z,t)-v_m(z,t))\leq
 3\epsilon.$$ 
 When $\epsilon\rightarrow 0$, we obtain
 $$u(z,t)\leq v(z,t).$$
 Since the roles of $u$ and $v$ are symmetric, $v(z,t)\leq u(z,t)$. Then $u=v$.
 \end{proof}
 \section{Further directions}
 In this section, we discuss further questions in the same general directions as our result.
 On compact K\"ahler manifolds, the corresponding problem was solved in the case where 
 $f=0$ and $u_0$ has zero Lelong numbers. In that case, there exists a solution $u$ satisfying 
 $u(.,t)\rightarrow u_0$ in $L^1$ (see \cite{GZ13}), and the solution is unique
 (see  \cite{DL14}). 
 It is natural to ask whether the same result holds  
 for a domain in $\C^n$. Let us state our conjecture
 \begin{Conj}
 If we replace the condition "$u_0\in L^{\infty}(\tilde{\Omega})$" in Theorem \ref{maintheo}
 by the condition "$u_0$ has zero Lelong numbers" then there exists a unique function 
 $u\in C^{\infty}(\bar{\Omega}\times (0,T))$ satisfying \eqref{psh.cmdlc},
 \eqref{KRF.cmdlc}, \eqref{bien.cmdlc} such that $u(.,t)\rightarrow u_0$ in $L^1(\Omega)$.   
 \end{Conj}
 
 The case where $u_0$ has positive Lelong numbers is another problem. It was also
 considered and solved  in the case compact K\"ahler manifold by \cite{GZ13} and
 \cite{DL14}. It is the motivation of the second direction: the case of domain in
 $\C^n$  and $u_0$ has positive Lelong numbers.
 
 There is another question: What is the behavior when we replace the condition 
 "$u_0\in PSH(\tilde{\Omega})$" in Theorem \ref{maintheo} by the condition 
 "$u_0\in PSH(\Omega)$"? In order to prove Theorem \ref{maintheo}, we construct  
 plurisubharmonic functions $u_{0,m}$ which converge to $u_0$.
  This step is easy if we suppose that
 $u_0\in PSH(\tilde{\Omega})$. If we only suppose that "$u_0\in PSH(\Omega)\mbox{ and } 
 \lim\limits_{z\rightarrow z_0\in\partial\Omega}u_0(z)=\varphi (z_0)$", maybe this step is 
 still realizable but more difficult. We give a provisional result in this direction. 
 \begin{Prop}
 Let $\Omega$ be a bounded smooth strictly pseudoconvex domain of $\C^n$ and $T\in (0,\infty ]$.
 Let $u_0$ be a continuous plurisubharmonic function on $\Omega$ such that $u_0$ is smooth on 
 $\bar{\Omega}\setminus \mathcal{K}$, where $\mathcal{K}\Subset\Omega$. Assume that $\varphi, 
 f$ are functions satisfying the conditions of Theorem \ref{maintheo}. Then there exists 
  a unique function 
 $u\in C^{\infty}(\bar{\Omega}\times (0,T))\cap C(\bar{\Omega}\times [0,T))$ satisfying 
 \eqref{psh.cmdlc},  \eqref{KRF.cmdlc}, \eqref{bien.cmdlc} 
 and $u(.,0)=u_0$.
 \end{Prop}
 \begin{proof}[Proof sketch.]
 Let $\rho$, $\zeta$ be the functions defined in the proof of Theorem \ref{maintheo}. Let $\psi$ be a
 smooth function in $\Omega$ and $\phi$ be a smooth function on $\R$ satisfying
 \begin{itemize}
 \item $0\leq \psi\leq 1$, $\psi|_{U_1}=1$,$\psi|_{\Omega\setminus U_2}=0$, where 
 $\mathcal{K}\Subset U_1 \Subset U_2 \Subset \Omega$.\\
 \item $\phi$ is convex and increasing, 
 $\phi|_{(-\infty,-3)}=-2$, $\phi|_{(-1,\infty)}=Id$.\\
 \end{itemize}
  Using convolutions of $u_0+\frac{\rho}{m}$, we can find 
  $\tilde{u}_{0,m}\in C^{\infty}(U_2)$ such that $\tilde{u}_{0,m}$ and $\psi\tilde{u}_{0,m}+
  (1-\psi)(u_0+\frac{\rho}{m})$ are strictly plurisubharmonic functions.\\ 
  We define $u_{0,m}\in C^{\infty}(\bar{\Omega})$, $g_m\in 
  C^{\infty}(\bar{\Omega}\setminus\mathcal{K})$, 
  $\varphi_m\in C^{\infty}(\bar{\Omega}\times [0,T))$ by
  $$u_{0,m}=\psi\tilde{u}_{0,m}+(1-\psi)(u_0+\frac{\rho}{m})+\frac{1}{m}\phi\circ (m\rho),$$
  $$g_m=-\dot{\varphi}|_{t=0}+\log\det(u_0+
 \dfrac{m+1}{m}\rho)_{\alpha\bar{\beta}}+f(t,z,u_0+\dfrac{m+1}{m}\rho),$$
 $$\varphi_m=(1-\psi)(t\zeta(mt)g_m+u_0+\dfrac{m+1}{m}\rho+\int\limits_0^t\dot{\varphi}
 ).$$
 Repeating the techniques in the proof of Theorem \ref{maintheo}, we show that  there exists 
  a unique function 
 $u\in C^{\infty}(\bar{\Omega}\times (0,T))\cap C(\bar{\Omega}\times [0,T))$ satisfying 
 \eqref{psh.cmdlc},  \eqref{KRF.cmdlc}, \eqref{bien.cmdlc} 
 such that $u|_{t=0}=u_0$.\\
 \end{proof}

\end{document}